\documentclass[a4paper,twoside,10pt]{article}

\usepackage[T1]{fontenc}
\usepackage{times}
\usepackage{amsmath,amsfonts,amssymb,amsthm,euscript,pifont}
\usepackage{mathrsfs}

\usepackage[colorlinks=true]{hyperref}
\usepackage{fullpage}

\usepackage{color}
\usepackage{xspace}
\usepackage{graphicx} 
\usepackage{latexsym}
\usepackage{subfigure}
\usepackage{bbm}
\usepackage{enumitem}
\usepackage[squaren]{SIunits}

\numberwithin{equation}{section}
\def\PP{\mathbb{P}}

\newcommand{\epsin}{\varepsilon_{\text{in}}}
\newcommand{\epsout}{\varepsilon_{\text{out}}}

\newtheorem{thm}{Theorem}

\newtheorem{cor}[thm]{Corollary}

\newtheorem{rk}[thm]{Remark}

\DeclareMathOperator*{\argmin}{argmin}

\title{Monte Carlo methods for linear and non-linear Poisson-Boltzmann equation\thanks{This work was carried out and financed within
  the framework of the WalkOnMars project of the CEMRACS 2013.}}
\author{Mireille Bossy\thanks{INRIA Sophia Antipolis -- M\'editerran\'ee, TOSCA project-team,  2004 route des Lucioles, BP.\ 93,  06902 Sophia Antipolis Cedex, France; mireille.bossy@inria.fr} 
\and Nicolas Champagnat\thanks{Universit\'e de Lorraine, Institut Elie Cartan de Lorraine, UMR 7502, Vand\oe uvre-l\`es-Nancy,
  F-54506, France; Nicolas.Champagnat@inria.fr, CNRS, Institut Elie Cartan de Lorraine, UMR 7502, Vand\oe uvre-l\`es-Nancy,
  F-54506, France, INRIA, TOSCA project-team,, Villers-l\`es-Nancy, F-54600, France} 
\and {H\'el\`ene Leman}\thanks{CMAP, Ecole Polytechnique, CNRS, route de Saclay, 91128 Palaiseau Cedex-France}
\and {Sylvain Maire}\thanks{Aix-Marseille Universit\'e, CNRS, ENSAM, SIS, UMR 7296, F-13397 Marseille, Universit\'e
  de Toulon, CNRS, LSIS, UMR 7296, F-83957 La Garde; maire@univ-tln.fr}
\and Laurent Violeau\thanks{INRIA Sophia Antipolis -- M\'editerran\'ee, TOSCA project-team,  2004 route des Lucioles, BP.\ 93,  06902 Sophia Antipolis Cedex, France}
\and {Mariette Yvinec}\thanks{INRIA Sophia Antipolis -- M\'editerran\'ee, GEOMETRICA  project-team, 2004 route des Lucioles, BP.\ 93,
  06902 Sophia Antipolis Cedex, France} 
  }

\begin{document}

\maketitle

\begin{abstract} The electrostatic potential in the neighborhood of a biomolecule can be computed thanks to the non-linear divergence-form elliptic Poisson-Boltzmann PDE. 
Dedicated Monte-Carlo methods have been developed   to solve its  linearized version (see e.g.\cite{bossy-champagnat-al-09},\cite{mascagni-simonov-04}). These algorithms combine {\it walk on spheres} techniques and appropriate replacements at the boundary of the molecule. 
In the first part of this article we compare  recent replacement methods for this linearized  equation on real size biomolecules, that also require  efficient computational geometry algorithms. We compare our results with the deterministic solver APBS. 
In the second part, we prove a new probabilistic interpretation of the nonlinear Poisson-Boltzmann PDE.  A Monte Carlo algorithm is also derived  and tested on a simple test case.
\end{abstract}

\section{Introduction}
The goal of this paper is to study Monte Carlo methods to solve both linear and nonlinear versions of the Poisson-Boltzmann PDE.  In the linear case, we implement these algorithms on real size  biomolecules, the geometrical complexity being  handled thanks to efficient computational geometry algorithms. 

The Poisson-Boltzmann equation describes the electrostatic potential of a biomolecular system in a ionic solution with permittivity
$\varepsilon$, assuming a mean-field distribution of solvent ions---an assumption known as the {\it implicit solvent} 
approximation~\cite{baker-bashford-al-05}. In this work, we are interested in the case of a neutral ionic solution with two kinds of
ions with opposite charge, where the Poisson-Boltzmann equation takes the form
\begin{equation}\label{eq:nonlinearizedPB}
   -\nabla \cdot(\varepsilon(x) \nabla u(x))+ \kappa^2(x) \sinh(u(x))=f(x), \quad \forall x \in \mathbb{R}^3,
\end{equation}
where
\begin{equation}  \label{eq:def-eps-kappa}
  \varepsilon(x)=
  \begin{cases}
    \epsin >0 & \mbox{if\ }x\in\Omega_{\text{in}}, \\
    \epsout >0 & \mbox{if\ }x\in\Omega_{\text{out}},
  \end{cases}\qquad
  \kappa(x)=
  \begin{cases}
    0 & \mbox{if\ }x\in\Omega_{\text{in}}, \\
    \bar\kappa:=\sqrt{\epsout}\  \kappa_{\text{out}}>0 & \mbox{if\ }x\in\Omega_{\text{out}}.
  \end{cases}
\end{equation}
Here, $\Omega_{\text{out}}$ represents the domain of the ionic solvent, and $\Omega_{\text{in}}=\Omega_{\text{out}}^c$ represents
the interior of the molecule, defined as the union of $N$ spheres of centers $c_1,c_2,\ldots,c_N$ and radii $r_1,r_2,\ldots,r_N$
respectively, representing the atoms of the molecule, that is 
$$\Omega_{\text{in}}=\bigcup_{i=1}^N S(c_i,r_i),$$
where $S(c,r)=\{x\in\mathbb{R}^3:|x-c|<r\}$. We denote by $\Gamma$ the boundary of the bounded domain $\Omega_{\text{in}}$. More
precisely, $\Gamma$ can be either the van der Waals surface when the $r_i$ are the radii of the atoms, or the Solvent Accessible
Surface (SAS), which is the set of spheres with radii $r_i = \rho_i+r_s$, where $r_s$ is the radius of solvent molecules and $\rho_i$
are the radii of the atoms. The positive numbers $\epsin$ and $\epsout$ are the relative permittivity of each medium, and $u(x)$
represents the dimensionless potential at $x\in\mathbb{R}^3$, defined as $u(x)=e_c (k_BT)^{-1} \Phi(x)$ where $\Phi$ is the potential
and the constants $e_c$, $k_B$ and $T$ are respectively the charge of an electron, Boltzmann's constant and the absolute temperature.
Some differences may be found in the normalizing constants depending on the bibliographic references. We focus here on the choice of
the normalization, the values of each constants and the derivation of Poisson-Boltzmann's equation given in ~\cite{holst-94}, also
used in the Poisson-Boltzmann solver APBS~\cite{baker-sept-al-01}.

The source term $f$ is given as a sum of Dirac measures:
\begin{equation}  \label{eq:terme-source}
  f(dx)=\sum_{i=1}^{N} \left( \dfrac{e_c^2}{k_B T \varepsilon_0} \right) z_i \delta_{c_i}(dx),
\end{equation}
where $\varepsilon_0$ represents the absolute permittivity of vacuum, and $z_i$ is the relative charge of the $i$-th atom of the
molecule (relative charge meaning its actual charge divided by $e_c$). Note that, even though $\varepsilon$ and $\kappa$ are discontinuous
and $f$ is a measure, a proper notion of solution can be found in~\cite{chen-holst-xu-07}. Note also that $\kappa$ is sometimes
considered discontinuous at the ion accessible surface (obtained as the SAS surface with $r_s$ replaced by the radius of ions in the
solvent) and $\varepsilon$ at the van der Waals or the SAS surface. In this work we present our methods for a single discontinuity
surface for both $\kappa$ and $\varepsilon$ (either Van-der-Walls or SAS), but the methods can easily treat the double discontinuity
surfaces model.

Finally, $\bar\kappa^{-1}$ is the Debye length in the ionic solution (see e.g.~\cite{holst-94}),
\begin{equation}  \label{eq:kappa}
   \bar\kappa^{-2}=\dfrac{2 \mathcal{N}_A e_c^2 I}{\epsout\varepsilon_{0} k_B T},
\end{equation}
where $\mathcal{N}_A$ is the Avogadro constant, and  $I=c\,z^2$ is the ionic strength of the solvent,  where $c$ is the concentration of one of the ion species in the ionic solution and $z$ its relative charge (we recall that the two ion species have the same concentration and opposite charges).  

The values of all the physical constants  used  in the simulations are reported in Appendix \ref{appendixA}. 

In structural biology, the Poisson-Boltzmann equation is often used in  its linearized form: 
\begin{equation}  \label{eq:linearizedPB}
  -\nabla(\varepsilon(x) \nabla u(x))+ \kappa^2(x) u(x)=f(x), \quad \forall x \in \mathbb{R}^3
\end{equation}
which gives a good approximation of the electrostatic potential around uncharged molecules. We refer to Folgari et
al.~\cite{fogolari-al-02} for a survey on applications in structural biology.

Monte Carlo (MC) algorithms for the linearized Poisson-Boltzmann equation were proposed first by Mascagni and Simonov in~\cite{mascagni-simonov-03,mascagni-simonov-04}
and improved or extended in~Mascagni et al. \cite{simonov-mascagni-al-07,simonov-08,mascagni-al-10,mascagni-al-13}, Bossy et al.~\cite{bossy-champagnat-al-09}. These MC algorithms involve double randomization techniques~\cite{mascagni-simonov-04}, walk on spheres techniques~\cite{sabelfeld-91} to simulate Brownian motion exit times and positions in $\Omega_{\text{out}}$ and $\Omega_{\text{in}}$, and asymmetric jump methods from the boundary $\Gamma$ to
take into account the discontinuity of $\varepsilon$ in the divergence form of the PDE \eqref{eq:linearizedPB}.  
Recently, Lejay and Maire~\cite{lejay-maire-13} and Maire and  Nguyen~\cite{maire-nguyen-13} have proposed  new replacement methods from the boundary $\Gamma$ which improve the order of convergence of the original algorithm. 

Section \ref{sec:linear} is devoted to MC algorithms for the linearized Poisson-Boltzmann equation. 
After recalling their  general forms, we compare these MC algorithms on biomolecule geometries, i.e.\ for domains $\Omega_{\text{in}}$ defined from measured biomolecular data. 
The key issue to deal with real-size biomolecule geometries consists in finding efficient algorithms to locate the closest atom from any position in $\mathbb{R}^3$. In this work, we use 
the efficient power diagram construction, search and exploration tools developed in the CGAL library~\cite{CGAL} to solve this specific problem.

Section \ref{sec:nonlinear} is devoted to the study of probabilistic interpretations and Monte Carlo methods for the nonlinear Poisson-Boltzmann equation. We show how branching versions of diffusion processes may be used to deal with the nonlinear version of the PDE (Sections~\ref{sec:PDE-BBM} and~\ref{sec:first-choice}). We  derive the corresponding  Monte Carlo algorithm in Sections~\ref{sec:algo} and~\ref{sec:algo2}. We test  the numerical method on simple molecule test-cases with one or two atoms (Section~\ref{sec:cas-test}).

\section{Linear case}\label{sec:linear}

This section deals with the linearized Poisson-Boltzmann equation \eqref{eq:linearizedPB}. The Monte-Carlo algorithm used to solve
this equation was first proposed by Mascagni and Simonov in \cite{mascagni-simonov-04}, and the probabilistic interpretation of the
PDE and improved replacement algorithms are given in Bossy et al~\cite{bossy-champagnat-al-09}. This last reference also presents numerical tests on cases where the molecule has one or two atoms.

In this work, the simulation code that we implement can deal with molecules having an arbitrary number of atoms. We use the PDB format files of biomolecules which can be found for example on the RCSB Protein Data Bank, and convert it into PQR files containing the positions, radii and charges of all the atoms of the biomolecule using PDB2PQR~\cite{PDB2PQR}. This gives all the parameters of the
Poisson-Boltzmann equation, except the Debye length $\bar{\kappa}^{-1}$, which is computed from~\eqref{eq:kappa} (see the Appendix \ref{appendixA} for the explicit values).
 
When designing a Monte-Carlo method for the linearized Poisson-Boltzmann equation, 
the main difficulty is to deal correctly with the discontinuous coefficient in the divergence form operator. The key point is the equivalent formulation of the equation~\eqref{eq:linearizedPB} as two subproblems
\begin{gather}
  \label{eq:PDE-1}
  \begin{cases}
    -\varepsilon_{\text{in}}\Delta u(x)=f(x) &
    \mbox{for\ }x\in \Omega_{\text{in}} \\
    u(x)=h(x) & \mbox{for\ }x\in\Gamma,
  \end{cases} \\
  \label{eq:PDE-2}
  \begin{cases}
    -\varepsilon_{\text{out}}\Delta u(x)+\bar\kappa^2 u(x)=0 &
    \mbox{for\ }x\in \Omega_{\text{out}} \\
    u(x)=h(x) & \mbox{for\ }x\in\Gamma,
  \end{cases}
\end{gather}
 with a transmission condition~(see e.g.\cite{ladyzhenskaya-uraltseva-68}), which holds true in general for smooth $\Gamma$ (see~\cite{bossy-champagnat-al-09}): 
\begin{equation}  \label{eq:raccordement}
\begin{aligned}
\mbox{$h = u|_{\Gamma}$, $u$ is continuous on $\mathbb{R}^3$, and }\\
  \varepsilon_{\textup{in}}\nabla_{\text{in}}u(x)\cdot
  n(x)=\varepsilon_{\textup{out}}\nabla_{\text{out}}u(x)\cdot n(x),\quad\forall x\in\Gamma,
  \end{aligned}
\end{equation}
where
\begin{equation*}
  \nabla_{\text{in}}
  u(x):=\lim_{y\in\Omega_{\textup{in}},\:y\rightarrow x}\nabla u(y),\qquad
  \nabla_{\text{out}}
  u(x):=\lim_{y\in\Omega_{\textup{out}},\:y\rightarrow x}\nabla\varphi(y),\quad\forall x\in\Gamma
\end{equation*}
and $n(x)$ is the normal vector to $\Gamma$ at $x\in\Gamma$ pointing towards $\Omega_{\text{out}}$.

\subsection{A general probabilistic interpretation for \eqref{eq:PDE-1}-\eqref{eq:PDE-2}
-\eqref{eq:raccordement}}

The general principle of the Monte Carlo algorithms used here is given by the following extended Feynman-Kac formula for the solution
$u$ of \eqref{eq:linearizedPB} (see \cite{bossy-champagnat-al-09}) $\forall x \in \mathbb{R}^d \setminus \{ c_1,\ldots, c_N\}$,
\begin{equation}  \label{eq:FK}
  u(x)=\mathbb{E}_x \left[ \sum_{k=1}^{+\infty} \Big(u_0({X}_{\tau_k}) - u_0({X}_{\tau'_k})\Big) \exp \Big( - \int_0^{\tau_k} \varepsilon_{\text{out}} \kappa^2({X}_s)ds\Big)\right]
\end{equation}
where for all $h>0$, we denote $\Omega_{\text{in}}^h:=\{ x \in \Omega_{\text{in}}, d(x,\Gamma)\geq h\}$ and define $\tau'_0=0$ and
$\forall k \geq 1$,
\begin{align*}
  \tau_{k} &=\inf \{ t \geq \tau'_{k-1}, {X}_t \in \Omega_{\text{in}}^h \} \\
  \tau'_k &=\inf \{ t \geq \tau_{k}, {X}_t \in \Gamma \}.
\end{align*}
The function 
\begin{equation}\label{eq:def-u_0}
u_0(x):=\frac{e_c^2}{k_B T \varepsilon_0}  \sum \limits_{i=1}^N \dfrac{z_i}{4 \pi \epsin |x-c_i|}  
\end{equation}
is solution to $-\epsin \Delta u_0=f$ in $\mathbb{R}^3$, and $({X}_t, t \geq 0)$ is the weak solution to the stochastic
differential equation with weighted local time at the boundary $\Gamma$
\begin{equation*}
\left\{
 \begin{aligned}
  &{X}_t=x +\int_0^t \sqrt{2 \varepsilon({X}_{\theta})} dB_{\theta} + \frac{\varepsilon_{\text{out}}-\epsin}{2 \epsout} \int_0^t n({X}_{\theta}) dL^0_{\theta}(Y) \\
&Y_t \text{ is the signed distance of } {X}_t \text{ to } \Gamma \text{ (positive in } \Omega_{\text{in}} \text{)},\\
 & L^0(Y) \text{ is the local time at 0 of the semimartingale } Y.
 \end{aligned}
\right.
\end{equation*}
The formula \eqref{eq:FK} suggests to use a Monte Carlo approximation of $u$ based on $M$ independent simulations of $(X_t,t\geq 0)$. 
The MC algorithm computes the solution $u$ at specific points (as needed for the computation of the solvation free energy~\cite{baker-bashford-al-05}) and of the same PDEs with different parameters~\cite{simonov-mascagni-al-07}.

Moreover, the MC method takes  advantage of the geometry of the problem (the molecule is a union of spheres):
\begin{itemize}
\item away from $\Gamma$, since the paths of $X$ are scaled Brownian paths, we can use (centered and uncentered) walk on spheres techniques~\cite{sabelfeld-91} as 
fast numerical scheme. 
\item close to $\Gamma$, since~\eqref{eq:FK} only involves the position of $X_t$  and the amount of time spent in
$\Omega_{\text{out}}$ by $X$ between $\tau'_k$ and $\tau_k$ for all $k\geq 1$, we approximate $X$ by a process
that {\it jumps} away from $\Gamma$ when it hits $\Gamma$.
\end{itemize}

\subsection{The main steps of the Monte-Carlo algorithm}

Let us describe the steps of the simulation of $X$ for the Monte Carlo algorithms.
\subsubsection{Outside the molecule: the walk on spheres (WOS) algorithm}
\label{sec:outside}

Recall that in $\Omega_{\text{out}}$, $u$ is a solution to $ -\frac{1}{2}\Delta u+\lambda u=0$ with $\lambda:=\kappa^2_{\text{out}}/2$. Therefore, $u(x)=\mathbb{E}_x [u|_{\Gamma}(B_{\tau})e^{-\lambda\tau}]$, where under $\mathbb{P}_x$, $(B_t)_{t \geq 0}$ is a Brownian motion started at $x$ and $\tau$ is its first hitting time of $\Gamma$.

It is well-known~\cite{sabelfeld-91,mascagni-simonov-04,bossy-champagnat-al-09,lejay-maire-13} that the WOS algorithm simulates exactly successive positions of $B$ in $\Omega_{\text{out}}$ until it reaches --a small neighborhood of--  $\Gamma$, taking into account the exponential term in the probabilistic interpretation by means of a constant rate of killing $\lambda$.

Starting at a point $y_0:=x\in\Omega_{\text{out}}$, we find the largest sphere $S(y_0,r_0)$ included in $\Omega_{\text{out}}$ and centered at $y_0$. In that sphere, the killed Brownian motion either dies before reaching the boundary of $S(y_0,r_0)$ with probability $1- r_0 \sqrt{2\lambda}/\sinh (r_0 \sqrt{2\lambda})$, or it reaches the sphere boundary   at a point $y_1$ uniformly distributed on $S(y_0,r_0)$. We then start again the same procedure from $y_1$, and we obtain thus a sequence $(y_k)_{k\geq 0}$, possibly killed. Except in very specific situations, this sequence will a.s. never hit $\Gamma$ with a finite number of steps. Hence it is classical to introduce a small parameter $\epsilon>0$ and to stop the algorithm either at the first killing or at the first step where $y_k$ is in the $\epsilon$-neighborhood of $\Gamma$ and project $y_k$ on $\Gamma$ to obtain an approximation of the exit point. More formally, this algorithm can be written as follows.

\medskip
{\tt
\begin{enumerate}[noitemsep]
\item[] {\bf WOS algorithm in the domain $\Omega_{\text{out}}$.}
\item[] Set $k=0$. Given $y_0\in \Omega_{\text{out}}$, $\lambda\geq 0$, and
$\varepsilon >0$
\item Let $S(y_k,r_k)$ be the largest open sphere included in $\Omega_{\text{out}}$
  centered at $y_k$.
\item Kill the particle with probability
  $1-r_k\sqrt{2\lambda}/\sinh(r_k\sqrt{2\lambda})$, and goto END  if
killed.
\item  Sample $y_{k+1}$ according to the uniform distribution on
  $\partial S(y_k,r_k)$.
\item IF $d(y_{k+1},\partial \Omega_{\text{out}})\leq\varepsilon$, THEN set $\textit{exit\_position}(y_0)$ as the
  closest point of $\partial \Omega_{\text{out}}$ from $y_{k+1}$ and goto END.\\
  ELSE, set $k=k+1$ and return to Step~(1).
\item[] END. 
\end{enumerate}
}

\medskip
At least for smooth $\Gamma$, it is known~\cite{sabelfeld-91} that this algorithm stops a.s. in finite time, after a mean number of steps
of order $O(|\log(\epsilon)|)$. Moreover when $u$ is continuous on $\Gamma$,
\begin{equation}\label{eq:FK-outside}
  \mathbb{E}[u(\textit{exit\_position}(y_0))\mathbbm{1}_{\{\text{exit before killing}\}}]=\mathbb{E}_{y_0} [u(B_{\tau})e^{-\lambda\tau}]+O(\epsilon).  
\end{equation}

\subsubsection{CGAL Library: search for the closest atom}
\label{sec:cgal}

Given a position $y\in \Omega_{\text{out}}$, the first step of the WOS algorithm requires to construct the biggest open sphere
$S=S(y,r)$ in $\Omega_{\text{out}}$ with center $y$. In other words, it requires to find the nearest atom of the molecule from $y$,
i.e. the atom with index $\argmin_{1\leq i \leq N} (\|y-c_i\|-r_i)$, where $\|\cdot\|$ is the Euclidean norm in $\mathbb{R}^3$.

The simplest way to achieve such a search consists in doing a ``brute force'' naive search among the $N$ atoms to find the closest,
with a computational cost of order $N$. However, for such a minimization problems, it is generally possible to construct search algorithms with computational cost of order $\log N$, if one can afford to build an appropriate search tree in a precomputation step. Since $N$ could be large for biomolecules (several thousands) and since our algorithm requires to search for the closest atom a very large number of times (of order $\log\epsilon$ times for each independent simulation of the Monte-Carlo method), it is clearly interesting for us to use this second method.

The C++ library CGAL (Computational Geometry Algorithms Library)~\cite{CGAL} proposes geometric algorithms allowing to solve this problem. The idea is to construct first the power diagram associated to the set of spheres $(S(c_i,r_i))_{1\leq i \leq N}$, i.e. the partition of the space into polygonal cells, each of which are associated to an index $j\in\{1,\ldots,N\}$, such that the power distance $\|x-c_i\|^2-r_i^2$, $1\leq i\leq N$, of points $x$ of the cell, is minimal for $i=j$. Given a point $y\in\mathbb{R}^3$, the library can then easily compute, with complexity $O(\left|\log(\epsilon)\right| )$, the index of the cell containing $y$. Since it minimizes the power distance, this index is not necessarily the one that minimizes the Euclidean distance, so we need to check that none of the neighboring cells are closer. CGAL also provides tools to explore the neighboring cells in the power diagram, which we used. This last step is needed since otherwise one could use the WOS algorithm on a too large sphere and obtain points that would belong to $\Omega_{\text{in}}$. We ran several tests to find the proportion of such  events. Without the local search step, this proportion is
roughly $8\%$, whereas with the local search step, it drops down to roughly $2\%$\footnote{These proportions of course depend on the   biomolecule and on the starting point for the simulation of $X$; the values given here were obtained from a molecule with 103 atoms described in section \ref{sec:num-lin}.}. We use this local search in all the numerical tests presented here.

For comparison, we implement three methods to find the nearest atom from a point:
\begin{itemize}
\item The \texttt{brute localization} method is the naive but exact method, of complexity $O(N)$.
\item The \texttt{power diagram} method uses the CGAL tools described above, of complexity $O(\log N)$.
\item The \texttt{power diagram with hint}: 
in the WOS algorithm, the closest   atom in the previous step was already computed. Therefore, it can be used as a hint to find the atom with minimal power distance   among the neighboring atoms of the previous one. CGAL proposes options enabling to start the search from specific points. This method   is also of complexity $O(\log N)$, but hopefully with a smaller constant multiplying $\log N$.
\end{itemize}

The computational costs of those three methods are reported in Figure~\ref{fig:locatemethods}. Note that other ideas to improve the computational speed of this localization step were recently developed in~\cite{mascagni-al-13}.

\subsubsection{Inside the molecule: uncentered walk on spheres (UWOS) algorithm}
\label{sec:inside}

If the current position of the particle belongs to $\Omega_{\text{in}}$, we use the function $u_0$ in ~\eqref{eq:def-u_0} as the unique bounded solution of the PDE~\eqref{eq:PDE-1} in the domain $\mathbb{R}^3$. Thus $u-u_0$ is harmonic in $\Omega_{\text{in}}$, and using the notation of Section~\ref{sec:outside}, 
  for all $x\in\Omega_{\text{in}}$,
\begin{equation}  \label{eq:FK-inside}
  (u-u_0)(x)=\mathbb{E}\left[(u|_{\Gamma}-u_0)(B_{\tau})\right]. 
\end{equation}

Again, one can use a walk on spheres algorithm to compute this expectation. This can be done by taking advantage of the union of spheres geometry of the molecule. In this case, it is convenient to use an uncentered walk on spheres method: at each step we use the atom sphere to which the simulated path's position belongs, instead of drawing a virtual sphere centered around the current position.  The exit position from the sphere is not uniformly distributed on the sphere, but can be explicitly computed and exactly simulated~\cite{mascagni-simonov-04}. The following algorithm allows the exact simulation of the  exit position of a Brownian motion from $\Omega_{\text{in}}$.
\medskip

{\tt 
\begin{enumerate}[noitemsep]
\item[]{\bf UWOS algorithm in the domain $\Omega_{\text{in}}$.}
\item[] Set $k=0$. Given $y_0\in \Omega_{\text{in}}$,
\item Choose $i\in\{1,\ldots,n\}$ such that $y_k\in S(c_i,r_i)$.
\item Simulate $y_{k+1}=(r_i,\theta,\varphi)$ where $\theta$ is
  uniform on $[0,2\pi]$ and $\varphi$ is independent of $\theta$ with
  cumu\-la\-tive distribution function $F_{r_i,|y_k-c_i|}$, in the
  spherical coordinates centered at $c_i$ such that
  $y_k=(|y_k-c_i|,0,0)$.
\item IF $y_{k+1}\in\partial \Omega_{\text{in}}$, THEN set $exit(y_0)=y_{k+1}$ and goto
END.\\
ELSE, set $k=k+1$ and return to Step~(1).
\item[]END.
\end{enumerate}
}
\medskip

The cumulative distribution function $F_{R,r}$ is explicitly invertible and is given by
$$
F_{R,r}(\alpha):=\frac{R^2-r^2}{2Rr}
\left(\frac{R}{R-r}-
  \frac{R}{\sqrt{R^2-2Rr\cos\alpha+r^2}}\right).
$$

\subsubsection{On the boundary of the molecule: the jump method}
\label{sec:boundary}

The last ingredient of the MC algorithm is the discretization procedure to apply when the process $X$ hits the boundary $\Gamma$.  
We use approximations of the process $X$ that jumps immediately
after hitting $\Gamma$ either in $\Omega_{\text{in}}$ at a distance $h$ from $\Gamma$ (as in the probabilistic
representation of the solution $u$ in~\eqref{eq:FK}), or in $\Omega_{\text{out}}$ at a distance $\alpha h$ from $\Gamma$, for some constant $\alpha>0$. More
formally, we associate to each $x\in\Gamma$ a random variable $p(x)$ a.s. belonging to
$(\mathbb{R}^3\setminus\Gamma)\cup\{\partial\}$,  distributing the new position in $\mathbb{R}^3\setminus\Gamma$ of the process after its jump, or killing the process when it belongs to the cemetery point $\partial$. 

\medskip
The following algorithm computes a score along the trajectory of an approximation of $(X_t)_{t\geq 0}$.  In view of the probabilistic representation of $u$ ~\eqref{eq:FK}, the Monte-Carlo average of this score approximates $u(x_0)$.
\medskip

{\tt
\begin{enumerate}[noitemsep]
\item[]Given $x_0\not\in\{c_1,\ldots,c_N\}$, set $k=0$ and $score=u_0(x_0)$ if $x_0\in\Omega_{\text{in}}$ or $score=0$ otherwise. 
\item IF $x_k\in\Omega_{\text{in}}$,
  \begin{enumerate}[noitemsep]
  \item THEN use the UWOS algorithm to simulate $\textit{exit\_position}(x_k)$ and set
    $score =score - u_0(\textit{exit\_position}(x_k))$,
  \item ELSE use the WOS algorithm with
    $\lambda=\bar{\kappa}^2/2\varepsilon_{\text{out}}$ to simulate $\textit{exit\_position}(x_k)$.\\
    IF the particle has been killed, THEN return $score$ and goto END. 
  \end{enumerate}
\item Set $x_{k+1}$ equal to an independent copy of $p(\textit{exit\_position}(x_k))$.
\item IF $x_{k+1}\in\Omega_{\text{in}}$, THEN set $score =score +u_0(x_{k+1})$.
\item Set $k=k+1$ and return to Step~(1).
\item[]END.
\end{enumerate}
}
\medskip

In this work, we consider three different jump methods, i.e. three different families of r.v. $(p(x))_{x\in\Gamma}$. All are based on a finite difference approximation of the transmission condition~\eqref{eq:raccordement}. Note that other types of jump methods have been studied in the literature, among which ``jump on spheres'' techniques~\cite{simonov-08} and neutron transport approximations~\cite{bossy-champagnat-al-09}. For the first two methods, we follow the terminology of~\cite{bossy-champagnat-al-09}.

\medskip
\noindent \textbf{Symmetric normal jump (SNJ)}: 
This method is the one proposed by Mascagni and Simonov in their seminal paper~\cite{mascagni-simonov-04}. It can be justified by a first-order expansion in~(\ref{eq:raccordement}): for all $x\in\Gamma$,
$$
u(x)=\frac{\varepsilon_{\text{out}}}{\varepsilon_{\text{in}}+\varepsilon_{\text{out}}}
u(x+hn(x))+
\frac{\varepsilon_{\text{in}}}{\varepsilon_{\text{in}}+\varepsilon_{\text{out}}}u(x-h \, n(x))
+\mbox{ remainder}, 
$$
where the remainder is of order  $O(h^2)$ provided that the solution $u$ to the Poisson-Boltzmann equation has uniformly bounded second-order derivatives in $\Omega_{\text{out}}$ and $\Omega_{\text{in}}$. This holds true at least if $\Gamma$ is a $C^\infty$ manifold~\cite[Thm.\,2.17]{bossy-champagnat-al-09}. The expansion can be written as an expectation involving a Bernoulli r.v.\ $B$ with parameter $\frac{\varepsilon_{\text{out}}}{\varepsilon_{\text{in}}+\varepsilon_{\text{out}}}$ as
\begin{equation}  \label{eq:jump}
  u(x)=\mathbb{E}[u(x+(2 B-1)hn(x))]+O(h^2).
\end{equation}
This suggests the following choice for the r.v. $p(x)$: fix $h>0$, then for all $x\in\Gamma$,
\begin{equation}
  \label{eq:def-p-SNJ}
  p_{\text{SNJ}}(x)=
  \begin{cases}
    \displaystyle x+h\, n(x) & \mbox{with probability\ }
    \displaystyle \dfrac{\varepsilon_{\text{out}}}{\varepsilon_{\text{in}}+\varepsilon_{\text{out}}} \\[0.3cm]  
    \displaystyle x-h\, n(x) & \mbox{with probability\ }
    \displaystyle \dfrac{\varepsilon_{\text{in}}}{\varepsilon_{\text{in}}+\varepsilon_{\text{out}}}.
  \end{cases}
\end{equation}
Note that the full simulation algorithm with the SNJ jump method (called the SNJ algorithm)  can also be obtained by successive iterations of the formulas~\eqref{eq:FK-inside},~\eqref{eq:FK-outside} and~\eqref{eq:jump}, as explained in~\cite{mascagni-simonov-04}). This suggests that an error of order $h^2$ accumulates at each time the discretized process hits $\Gamma$. Since this number of hitting times is of order $1/h$, this suggests a global error of order $h$. Taking into account the additional error in the WOS algorithm, one can actually prove for smooth $\Gamma$ that the error between $u(x)$ and the expectation of the score of the SNJ algorithm is of order $O(h+\epsilon/h)$ when $h,\varepsilon\rightarrow 0$~\cite[Thm.\,4.7]{bossy-champagnat-al-09}.
\medskip

\noindent \textbf{Asymmetric normal jump (ANJ)}:
This method, proposed in \cite{bossy-champagnat-al-09}, can also be deduced from the transmission condition \eqref{eq:raccordement}, by introducing different finite difference steps to approximate the interior and exterior gradients. We fix $h>0$ and introduce a fixed parameter $\alpha>0$. Then, for all $x\in\Gamma$, we set
\begin{equation}
\label{eq:def-p-ANJ}
 p_{\text{ANJ}}(x)=
\begin{cases}
  x+\alpha h \, n(x) & \mbox{with probability\
  }\dfrac{\varepsilon_{\text{out}}}{\varepsilon_{\text{out}}+\alpha\varepsilon_{\text{in}}} \\[0.4cm]  
  x- h \, n(x) & \mbox{with probability\ }\dfrac{\alpha\varepsilon_{\text{in}}}
  {\varepsilon_{\text{out}}+\alpha\varepsilon_{\text{in}}}.
\end{cases}
\end{equation}
The error of the ANJ algorithm obtained with this jump method is also of order $O(h+\epsilon/h)$~\cite[Thm.\,4.7]{bossy-champagnat-al-09}, but, if $\alpha>1$, the process is moved further away from $\Gamma$ when it jumps in $\Omega_{\text{out}}$. Since the process is killed with a larger probability when it starts in $\Omega_{\text{out}}$ further away from $\Gamma$, assuming $\alpha>1$ makes the computational cost of a simulation score smaller than for the SNJ algorithm. Of course, a compromise must be found with the increased bias for increased $\alpha>1$, which is analysed in~\cite{bossy-champagnat-al-09}.

\medskip
\noindent \textbf{Totally Asymmetric Jump (TAJ)}: This jump method from $\Gamma$ is a new proposition in the context of the
Poisson-Boltzmann equation. It was originally  proposed in a two-dimensional context by Lejay and Maire \cite{lejay-maire-13} and for
more general equations and boundary conditions by Maire and Nguyen in \cite{maire-nguyen-13}. These methods apply to linear 
divergence form equations with damping. It is based on a higher order expansion of the transmission conditions of the PDE on $\Gamma$ and extensively use the  linearity. 
It does not apply in the nonlinear case, and will not be used in Section~\ref{sec:nonlinear}.

When the process hits the boundary, we replace it as follows: 
\begin{itemize}
\item with probability $\dfrac{\alpha \varepsilon_{\text{in}}}{\alpha
      \varepsilon_{\text{in}}+\varepsilon_{\text{out}}+\frac{1}{2}\bar\kappa^2\alpha^2h^2}$, the process moves toward
  $\Omega_{\text{in}}$ at one of the following four points with uniform probability: 
\begin{equation*}
\left\{x-hn(x)+\sqrt{2}hm(x), \; x-hn(x)+\sqrt{2}hq(x), \;  x-hn(x)-\sqrt{2}hm(x) \text{ or } x-hn(x)-\sqrt{2}hq(x)\right\}, 
\end{equation*}
 where $m(x)$ and $q(x)$ are any two orthonormal vectors in the tangent plane of $\Gamma$ at the point $x$,
 \item with probability $\dfrac{\varepsilon_{\text{out}}}{\alpha
       \varepsilon_{\text{in}}+\varepsilon_{\text{out}}+\frac{1}{2}\bar\kappa^2\alpha^2h^2}$, the process moves toward
   $\Omega_{\text{out}}$ at one of the four points with uniform probability: 
\begin{equation*}
\left\{x+\alpha  hn(x)+ \sqrt{2}\alpha hm(x), \; x+\alpha hn(x)+\sqrt{2}\alpha hq(x), \; x+\alpha hn(x)- \sqrt{2}\alpha hm(x) \text{ or }
 x+\alpha hn(x)-\sqrt{2}\alpha hq(x)\right\},
\end{equation*}
 \item with probability $\dfrac{\frac{1}{2}\bar\kappa^2\alpha^2 h^2}{\alpha
       \varepsilon_{\text{in}}+\varepsilon_{\text{out}}+\frac{1}{2}\bar\kappa^2\alpha^2h^2}$, the process is killed.
\end{itemize}
As stated in Theorem \ref{thm:cv-TAJ} below, the TAJ method is of order 2, whereas SNJ and ANJ are first order methods. 
\begin{thm}
  \label{thm:cv-TAJ}
  Assume that $\Gamma$ is a $C^\infty$ compact manifold in $\mathbb{R}^3$. Then, for all $x\not\in\{c_1,\ldots,c_N\}$, the
  expectation $\bar u_{h,\alpha,\epsilon}(x)$ of the score of the TAJ algorithm with parameters $h$, $\alpha$ and $\epsilon$, started
  from $x_0=x$, satisfies
  $$
  |\bar u(x)-u(x)|\leq C\left(h^2+\frac{\epsilon}{h}\right),
  $$
  for a constant $C$ depending only on $\alpha$ and the finite constant $\sup_{y\not\in\Gamma,\,\|y\|\leq R}(|u(y)|+\|\nabla u(y)\|+\|\nabla^2
  u(y)\|+\|\nabla^3 u(y)\|)$, for $R$ large enough such that $\Gamma\subset B(0,R)$, where $B(0,R)=\{z\in\mathbb{R}^3:\|z\|<R\}$.
\end{thm}
While the proof is based on similar computations in \cite{maire-nguyen-13}, for the sake of completeness we give a detailed proof in the context of Poisson Boltzmann equation. Note that the above TAJ replacement formulas are more convenient than
the formulas derived  in \cite{maire-nguyen-13}, as they do not require to impose some constraints on $h$. 

\begin{proof}
  In the case where $\Gamma$ is $C^\infty$, it has been proved in~\cite[Thm.\,2.17]{bossy-champagnat-al-09} that the solution $u$ to
  the Poisson-Boltzmann equation satisfies that $u_{\mid\Gamma}$ is $C^\infty$. Hence, the solutions of the two
  subproblems~(\ref{eq:PDE-1}) and~(\ref{eq:PDE-2}) admit derivatives of any order which are continuous up to $\Gamma$, i.e. they
  belong to $C^\infty(\overline{\Omega_{\text{in}}})$ and $C^\infty(\overline{\Omega_{\text{out}}})$, respectively
  (see~\cite{gilbarg-trudinger-01}). In particular, $\nabla^k u$ is bounded on $B(0,R)$ for all $k\geq 0$.

  Hence the following Taylor expansions are valid for all $x\in\Gamma$ and $y\in B(0,R)\setminus\Gamma$:
  \begin{gather}
    u(y)=u(x)+\nabla_{\text{in}}u(x)\cdot(y-x)+\frac{1}{2}(y-x)'\nabla^2_{\text{in}}u(x)(y-x)+O(\|y-x\|^3),\quad\text{ if
    }y\in\Omega_{\text{in}}, \label{eq:TAJ1}\\
    u(y)=u(x)+\nabla_{\text{out}}u(x)\cdot(y-x)+\frac{1}{2}(y-x)'\nabla^2_{\text{out}}u(x)(y-x)+O(\|y-x\|^3),\quad\text{ if
    }y\in\Omega_{\text{out}}, \label{eq:TAJ2}
  \end{gather}
  where the $O(\|y-x\|^3)$ are bounded by $\|y-x\|^3$ times a constant depending only on $\sup_{z\not\in\Gamma,\,\|z\|\leq
    R}\|\nabla^3 u(z)\|$, and where the notation $\nabla_{\text{in}}$ and $\nabla_{\text{out}}$ are extended to higher-order
  derivatives in an obvious way.

  Fix $x\in\Gamma$, $\eta\in\mathbb{R}$ and $\gamma>0$. Without loss of generality, we can assume that $x=0$, $n(x)=(1,0,0)$,
  $q(x)=(0,1,0)$ and $m(x)=(0,0,1)$. We define
  $$
  E^{\eta,\gamma}u=\frac{u(\eta,\gamma \eta, 0)+u(\eta,-\gamma \eta, 0)+u(\eta,0,\gamma \eta)+u(\eta,0,-\gamma \eta)}{4}.
  $$
  Applying~\eqref{eq:TAJ1}, we obtain
  $$
  E^{-h,\gamma}=u(0)-h\nabla_{\text{in}}u(0)\cdot n(0)+\frac{1}{4}\left[2h^2\frac{\partial^2_{\text{in}}u(0)}{\partial x^2}
    +\gamma^2h^2\left(\frac{\partial^2_{\text{in}}u(0)}{\partial y^2}+\frac{\partial^2_{\text{in}}u(0)}{\partial z^2}\right)\right]+O(h^3),
  $$
  and applying~\eqref{eq:TAJ2},
  $$
  E^{\alpha h,\alpha\gamma}=u(0)+\alpha h\nabla_{\text{out}}u(0)\cdot n(0)+\frac{1}{4}\left[2\alpha^2h^2\frac{\partial^2_{\text{out}}u(0)}{\partial x^2}
    +\gamma^2\alpha^2h^2\left(\frac{\partial^2_{\text{out}}u(0)}{\partial y^2}+\frac{\partial^2_{\text{out}}u(0)}{\partial z^2}\right)\right]+O(h^3).
  $$
  Now, the relations $\Delta u(x)=0$ in $\Omega_{\text{in}}$ in the neighborhood of $\Gamma$ and $\Delta u(x)=\kappa_{\text{out}}^2
  u(x)$ in $\Omega_{\text{out}}$ can be extended by continuity to $\Gamma$, so that
  $$
  \frac{\partial^2_{\text{in}}u(0)}{\partial y^2}+\frac{\partial^2_{\text{in}}u(0)}{\partial
    z^2}=-\frac{\partial^2_{\text{in}}u(0)}{\partial x^2},
  $$
  and
  $$
  \frac{\partial^2_{\text{out}}u(0)}{\partial y^2}+\frac{\partial^2_{\text{out}}u(0)}{\partial z^2}
  =-\frac{\partial^2_{\text{out}}u(0)}{\partial x^2}+\kappa_{\text{out}}^2 u(0).
  $$
  This entails
  $$
  E^{-h,\gamma}=u(0)-h\nabla_{\text{in}}u(0)\cdot n(0)+\frac{2-\gamma^2}{4}h^2\frac{\partial^2_{\text{in}}u(0)}{\partial x^2}+O(h^3)
  $$
  and
  $$
  E^{\alpha h,\alpha\gamma}=\left(1+\frac{\kappa_{\text{out}}^2\alpha^2\gamma^2h^2}{4}\right)u(0)+\alpha
  h\nabla_{\text{out}}u(0)\cdot n(0)
  +\frac{2-\gamma^2}{4}\alpha^2h^2\frac{\partial^2_{\text{out}}u(0)}{\partial x^2}+O(h^3).
  $$
  Hence, choosing $\gamma=\sqrt{2}$, we obtain
  \begin{equation}
    \label{eq:dl}
    \frac{\alpha\varepsilon_{\text{in}}}{\alpha\varepsilon_{\text{in}}+\varepsilon_{\text{out}}(1+\kappa_{\text{out}}^2\alpha^2h^2/2)}
    E^{-h,\sqrt{2}}
    +\frac{\varepsilon_{\text{out}}}{\alpha\varepsilon_{\text{in}}+\varepsilon_{\text{out}}(1+\kappa_{\text{out}}^2\alpha^2h^2/2)}E^{\alpha h,\sqrt{2}\alpha}
    =u(0)+O(h^3),
  \end{equation}
  where the gradient terms canceled because of~\eqref{eq:raccordement} and where the $O(h^3)$ is bounded by $h^3$ times a constant
  depending only on $\sup_{z\not\in\Gamma,\,\|z\|\leq R}\|\nabla^3 u(z)\|$. The TAJ jump method corresponds exactly to the
  probabilistic interpretation of this formula.

  Theorem~\ref{thm:cv-TAJ} then follows from~\eqref{eq:dl} exactly as Theorem~4.7 of~\cite{bossy-champagnat-al-09} follows from
  Equation~(4.19) of~\cite{bossy-champagnat-al-09}.
\end{proof}

\subsection{Numerical experiments}\label{sec:num-lin}

\subsubsection{Parallel version of the algorithm}
It is usually very simple to implement a parallel version of a Monte-Carlo algorithm.  This is the case for our algorithm. The only delicate issue for
reliable and statistically sound calculations is the parallel generation of pseudo random numbers. In our  MPI\footnote{Message Passing Interface (MPI) is a standardized and portable message-passing system available on a wide variety of parallel computers.} parallel implementation of the code, we used the 
SPRNG 4.4 library~\cite{sprng}. 

\subsubsection{Comparison of the `locate nearest atom' methods}

As described in Subsection \ref{sec:cgal}, we implemented three different methods to approximate the closest atom from a given
particle position in $\mathbb{R}^3$. In Figure \ref{fig:locatemethods}, we present the CPU time for each method as a function of the
size of the molecule. We use 5 molecules of different sizes: the molecule composed of $N=103$ atoms described in the next subsection,
and the molecules \texttt{2KAM}, \texttt{4HHF}, \texttt{1KDM}, \texttt{1HHO}, \texttt{1HFO} and \texttt{4K4Y} of RCSB Protein Data
Bank, with sizes ranging from $N=416$ to $N=29420$ atoms. We start each simulation from a location close to the alpha-carbon atom of
the first residue of the molecule. We use the SNJ
algorithm with $h=0.1$ and $\epsilon=10^{-4}$ (the shape of the curves is roughly independent on the jump method). We run $10^5$ independent simulations for each molecule on a laptop computer.  Of course, the
computational time is very dependent of the shape of the molecule and of the initial position of the algorithm, so that the CPU time
does not necessarily increase with $N$, as observed in Figure \ref{fig:locatemethods} for the largest molecules.

However, as expected, the \texttt{power diagram with hints} method is the fastest, at least for large molecules. The \texttt{power
  diagram} method is slightly slower, but the difference is not very significant. The \texttt{brute localization} method is up to 10
times slower than the \texttt{power diagram with hints} method for large molecules, but is actually faster for molecules of sizes
smaller than a thousand atoms.

\begin{figure}[ht]
  \centering
  \subfigure[CPU times (Log scale) in terms of the size of the molecule (Log scale)\label{fig:locatemethods}]{
    \includegraphics[width=.45\textwidth]{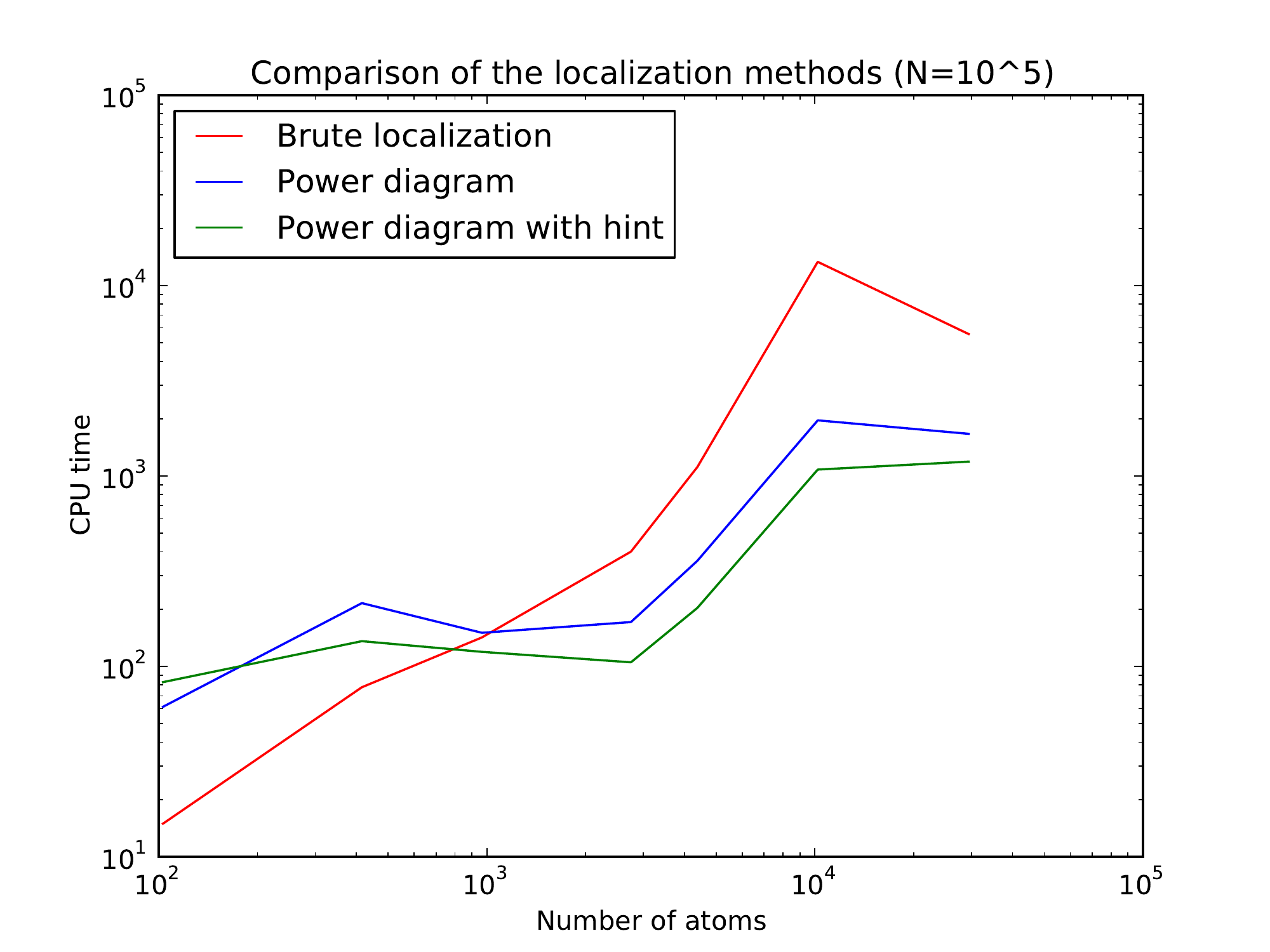}}\quad
  \subfigure[Comparison with APBS. The blue curve is produced with our code, and the red curve with APBS.\label{fig:APBS}]{
    \includegraphics[width=.45\textwidth]{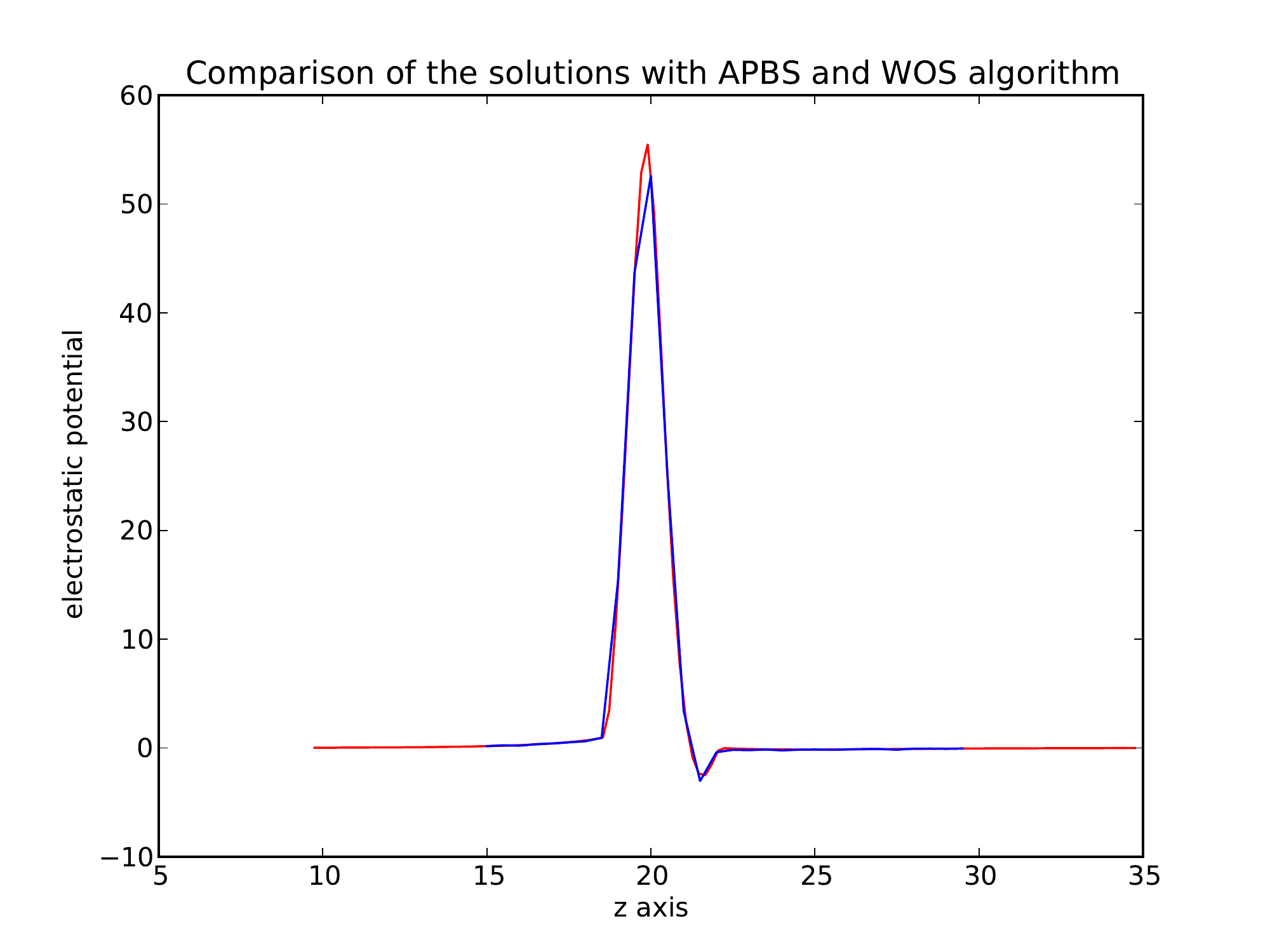}}
  \caption{Comparison of the localization methods and adequation with a deterministic method.}\label{figure1}
\end{figure}

\subsubsection{Comparison with APBS}\label{subsec:numAPBS}

As a validation of the algorithm, we compared our Monte-Carlo estimate of $u(x)$ with the value computed by a deterministic solver of
Poisson-Boltzmann equation. We choose APBS solver~\cite{baker-sept-joseph-al-01}, which uses adaptive finite element methods and
algebraic multilevel methods.

This numerical experiment is done on a small peptide composed of 6 residues (GLU-TRP-GLY-PRO-TRP-VAL) and $N=103$ atoms. To produce the plots in
Figure~\ref{fig:APBS}, we have calculated $u$ at 30 different points of the space located on a line close to the alpha-carbon of the
first residue (GLU). 
Our Monte-Carlo method was run with the SNJ jump method, $h=0.1$, $\epsilon=10^{-4}$ and $4\times 10^4$ independent simulations for each
initial points to compute the Monte-Carlo average. 
The agreement between the two curves is quite good, although there are some differences, which might be due either to the Monte-Carlo error, or to the discretization in APBS. 

\subsubsection{The TAJ methods' convergence order in the single atom case}

The goal of this experiment is to check that the expected error of the TAJ method converges faster than ANJ methods, as suggested by
Theorem~\ref{thm:cv-TAJ}. We used the simplest molecule, composed of a single atom, which is the only practical case where $\Gamma$
is $C^\infty$ in Poisson-Boltzmann equation. In~\cite{bossy-champagnat-al-09}, some comparisons were done on the  SNJ and ANJ convergences in  such a case. We consider an atom with radius $1\usk\angstrom$ and charge $1$. We compute the approximation of $(u-u_0)(x_0)$, where
$x_0$ is the center of the atom, using SNJ, ANJ ($\alpha=3$ and $\alpha=10$) and TAJ ($\alpha=1$, $\alpha=3$ and $\alpha=10$)
methods. The Monte-Carlo average is computed from $10^7$ independent simulations for each method and each values of $h$, ranging from
$0.003$ to $0.9$ and we took $\epsilon=10^{-5}$. The results are compared with the exact value of $(u-u_0)(x_0)$ for which an exact
expression is known~\cite{bossy-champagnat-al-09}.

\begin{figure}[ht]
  \centering
  \subfigure[Monte-Carlo average  as a function of $h$ (Log
  scale).\label{fig:average_lin_1_atom}]{\includegraphics[width=.45\textwidth]{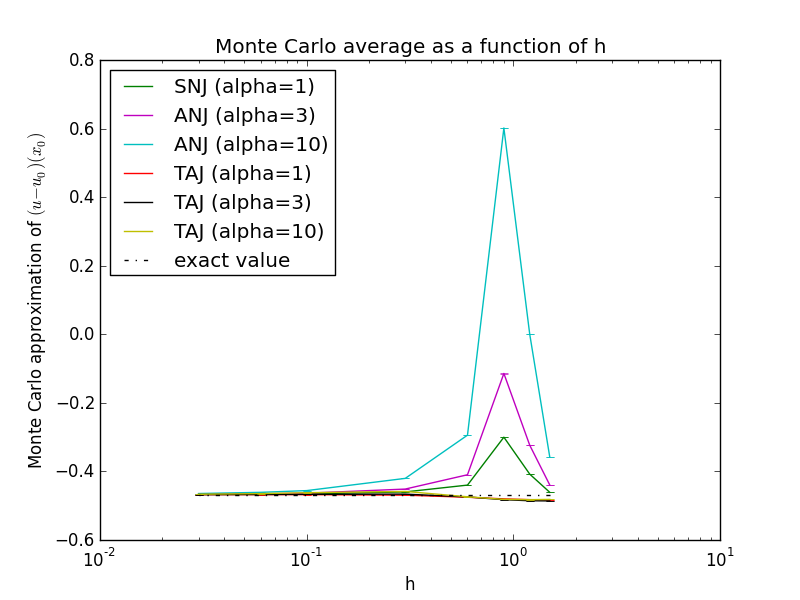}}\quad
  \subfigure[Error (Log scale)  as a function of $h$ (Log scale).\label{fig:error_lin_1_atom}]{
    \includegraphics[width=.45\textwidth]{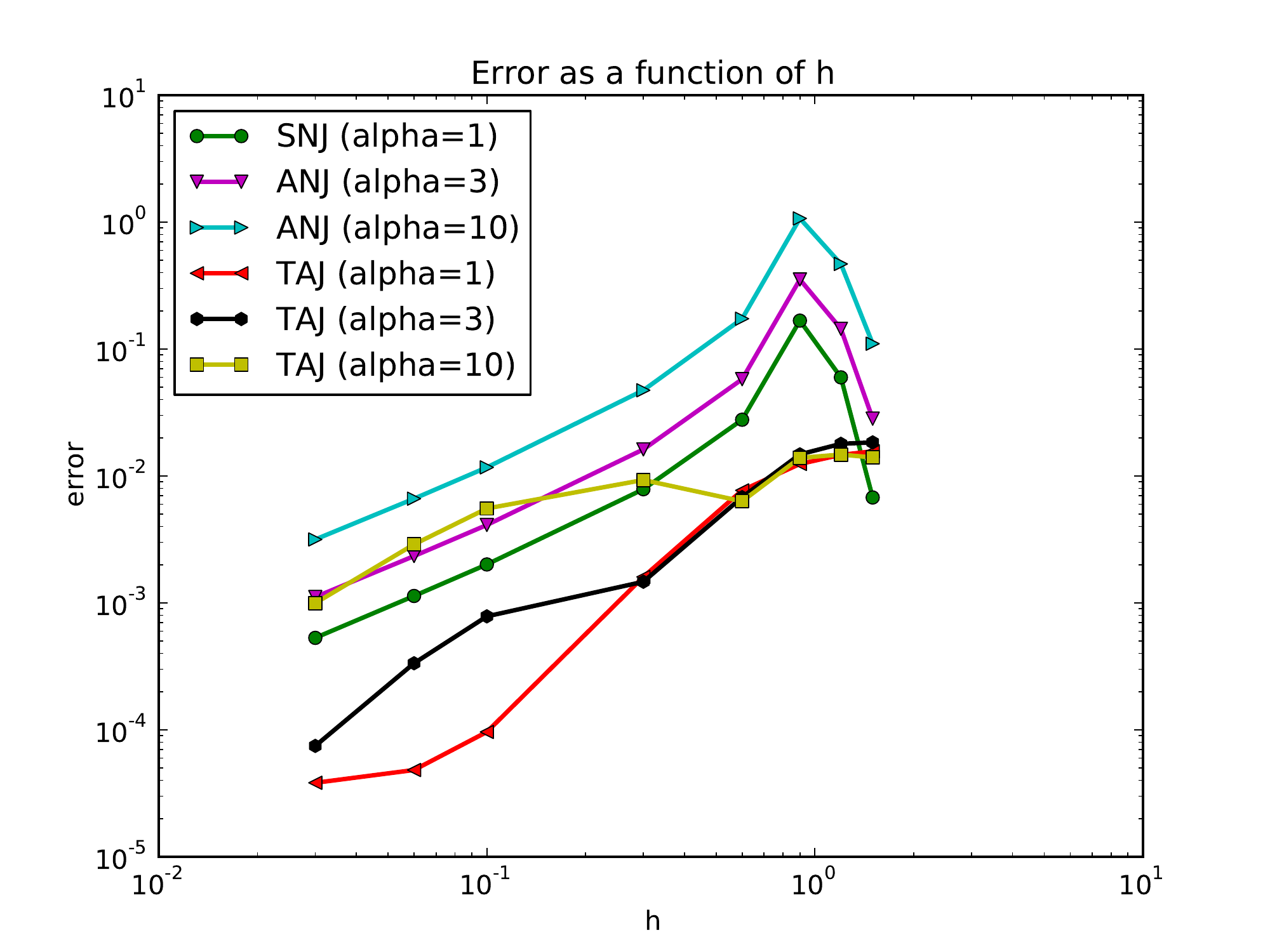}}
  \caption{Convergence and error of the linear with jump methods \texttt{SNJ}, \texttt{ANJ} ($\alpha=3$ and $\alpha=10$) and \texttt{TAJ} ($\alpha=1$, $\alpha=3$ and $\alpha=10$) for the single atom case.}\label{fig:TAJ-1-atom}
\end{figure}

The results are shown in Figure~\ref{fig:TAJ-1-atom} and Figure~\ref{fig:perform_lin_1_atom}. As expected, we observe a faster convergence for
the TAJ methods, although the second order is not very clear because of the Monte-Carlo error. The first order of convergence for the
SNJ and ANJ methods can be observed much more clearly. 

It seems that for a given $h$, the error is smaller for a smaller parameter $\alpha$ both for ANJ and TAJ methods, but this needs to
be compared with the CPU time of the simulations. The performance plot of Figure~\ref{fig:perform_lin_1_atom} shows the expected error
of the Monte-Carlo algorithm as a function of CPU time. It reveals that, for a given CPU time of computation and choosing an
appropriate value of $h$, the expected error of the algorithm is comparable for the three different values of $\alpha$, and is
slightly smaller for $\alpha=10$ for ANJ algorithms. This is consistent with the tests realized in~\cite{bossy-champagnat-al-09}.
This plot also confirms the better efficiency of the TAJ methods.

\subsubsection{Comparison between the different jump methods on a biomolecule}

The goal of this experiment is to compare the SNJ, ANJ and TAJ methods on a small molecule, but with realistic geometry. We use the
molecule composed of 103 atoms described in Section~\ref{subsec:numAPBS}. We use the jump methods SNJ, ANJ  ($\alpha=3$ and $\alpha=10$) and TAJ ($\alpha=1$, $\alpha=3$ and $\alpha=10$). The simulations are done with $\epsilon=10^{-6}$ and
with different values for $h$, ranging from $0.9$ to $0.003$, and we take the same number of Monte-Carlo simulations ($10^6$)
for each run, large enough to be able to detect and compare the convergence of the expectation of the score computed by our algorithms.

\begin{figure}[ht]
  \centering
  \subfigure[Monte-Carlo average  as a function of $h$ (Log
  scale).\label{fig:average_lin_eq7}]{\includegraphics[width=.45\textwidth]{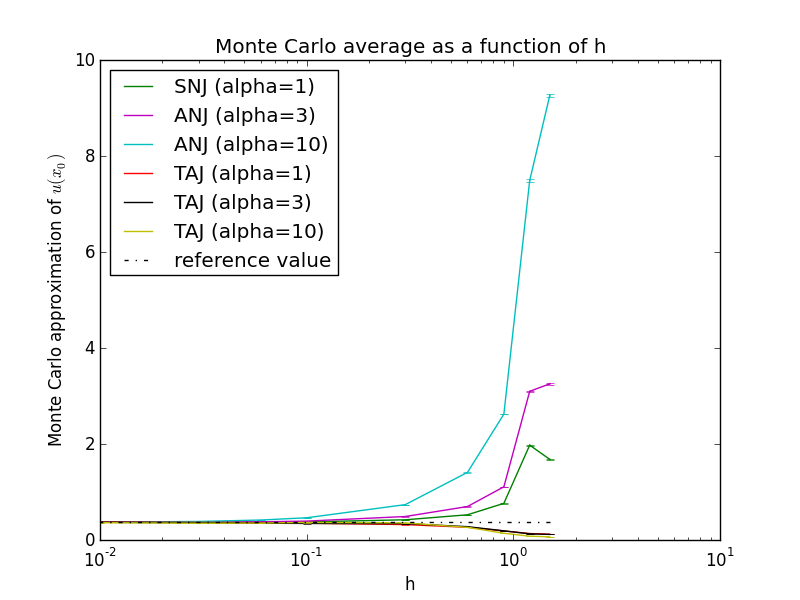}}\quad
  \subfigure[Error (Log scale)  as a function of $h$ (Log scale).\label{fig:error_lin_eq7}]{
    \includegraphics[width=.45\textwidth]{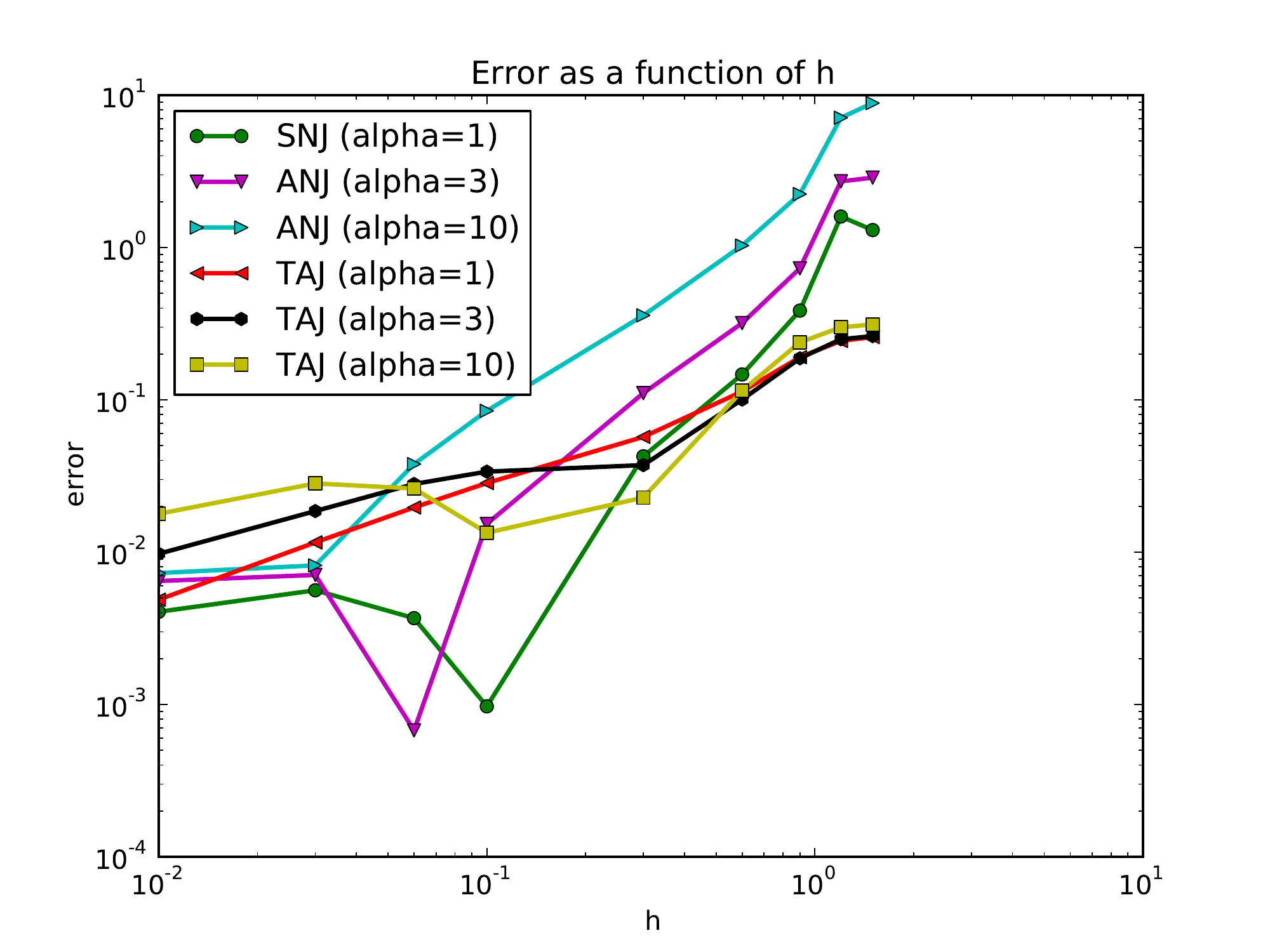}}
  \caption{Convergence and error of the linear with jump methods \texttt{SNJ}, \texttt{ANJ} ($\alpha=3$ and $\alpha=10$) and \texttt{TAJ} ($\alpha=1$, $\alpha=3$ and $\alpha=10$) for a molecule composed   of 103 atoms.}\label{fig:TAJ-eq7}
\end{figure}

Figure~\ref{fig:TAJ-eq7} shows the approximation of $u(x_0)$ by Monte-Carlo average, and the associated error relative to a reference value computed using ANJ ($\alpha=3$) algorithm, $h=0.001$, $\epsilon=10^{-7}$ and $(15\cdot10^6)$ runs for the Monte Carlo average. The point $x_0$ is chosen out to the molecule, but close (at $1\angstrom$) to the Trp amino acid. 

We also compared the result with the value
computed with the APBS method, but it differs from the Monte-Carlo reference value of roughly $7$\%, which is too large to detect the
rate of convergence for $h$ small. We suspect that this difference is due to the finite element discretization used in APBS.  

All 6 methods show a good convergence. The higher order of convergence of the TAJ method cannot be detected because of the
statistical error of the Monte Carlo method. However, we observe a smaller error for larger values of $h$ for the TAJ method than for
the ANJ ones. Of course, the TAJ methods might not converge with order 2 as in Theorem~\ref{thm:cv-TAJ}, because $\Gamma$ is not
smooth. This indeed also causes some difficulties in the implementation of the method, since it might occur, even for small $h$, that
a jump from $\Gamma$ to the direction of $\Omega_{\text{out}}$ actually gives a new position inside the molecule. This is
particularly true for the TAJ algorithm. Our tests show that the result of the algorithm is quite sensitive to the method used in this particular situation, and can produce differences  up to $5$\%. When this situation occurs, we chose here to push the particle
outside of the molecule, at a distance $h$ of $\Gamma$, close to its position before the jump. This choice shows good agreement
between the values given by the ANJ and TAJ algorithms.

As in the case of a single atom, we compare the performance of the 6 methods in Figure~\ref{fig:perform_lin_eq7}. This plot does not
show as clear conclusions as for the case of a single atom, but it confirms that TAJ has a better performance at small CPU times.

\begin{figure}[ht]
  \centering
  \subfigure[The error (Log scale) as a function of CPU time (Log
  scale), on a signe atom.\label{fig:perform_lin_1_atom}]{\includegraphics[width=.47\textwidth]{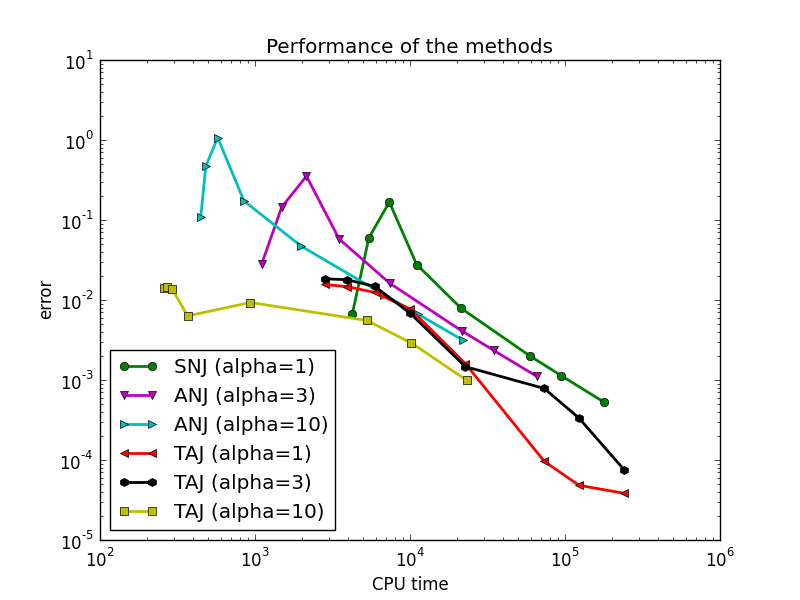}}\quad
  \subfigure[The error (Log scale) as a function of CPU time (Log
  scale), on a molecule  of 103 atoms.\label{fig:perform_lin_eq7}]{\includegraphics[width=.47\textwidth]{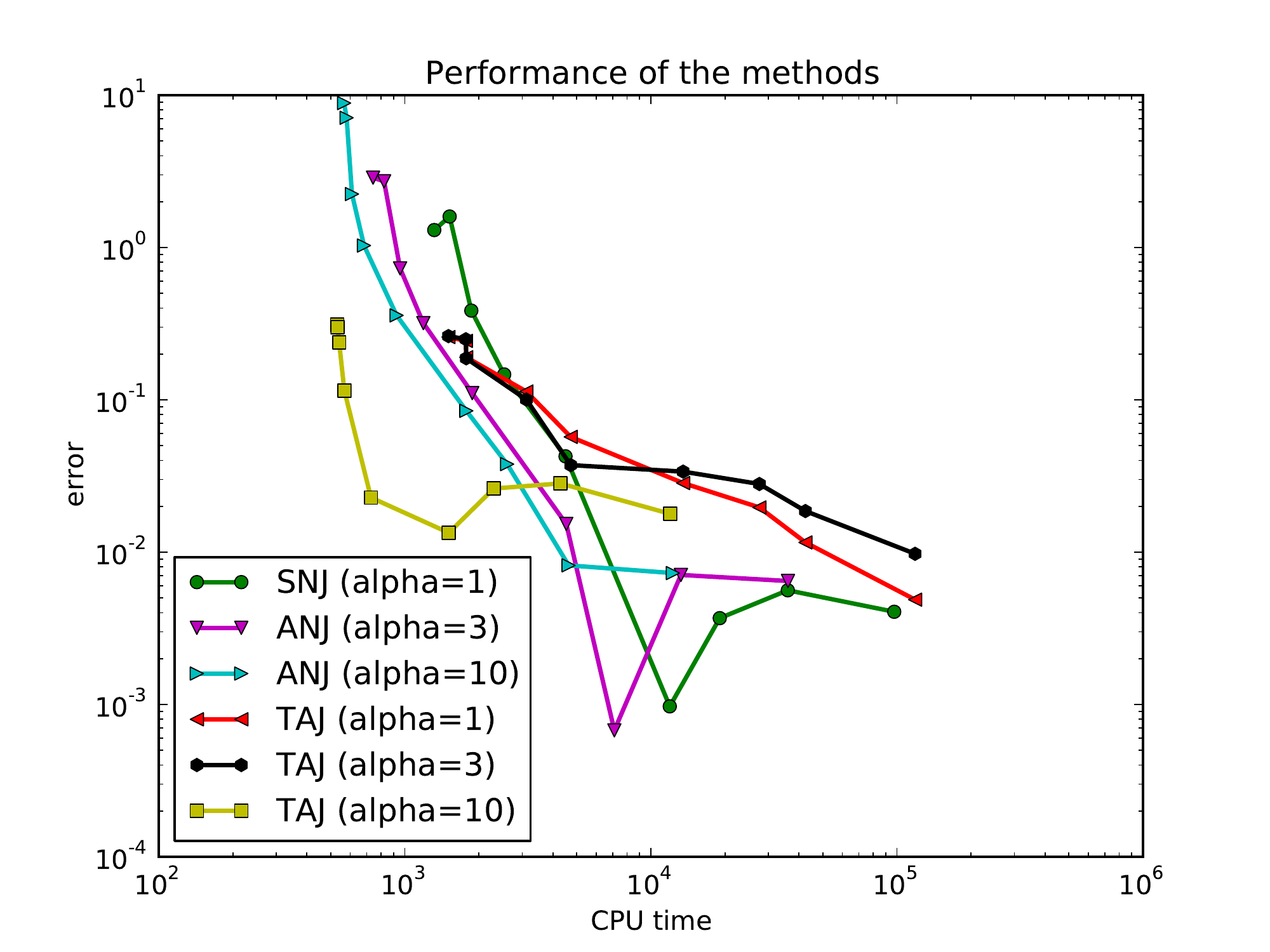}}
  \caption{Performance plots of the 6 linear with jump methods \texttt{SNJ}, \texttt{ANJ} and \texttt{TAJ} for a single atom (a), and a molecule composed of 103 atoms (b).}\label{fig:TAJ-both-perform}
\end{figure}
\section{Non-linear case}\label{sec:nonlinear}

The method we present here extends the previous probabilistic interpretation to the non-linear Poisson-Boltzmann
equation~\eqref{eq:nonlinearizedPB}, by making use of the link between PDEs with non-linearity of order 0 (in $u$) and branching
diffusions. Our algorithm is based on the simulation of a family of branching particles in $\mathbb{R}^3$. As in the linear case, our
method involves walk on spheres algorithms inside and outside the molecule, a jump and scoring procedure when the particle hits
$\Gamma$ (SNJ and ANJ jump methods), and a killing rate outside of the molecule. In the nonlinear case, our method also involves a possible reproduction of
a particle when it dies. This modifies the scoring procedure which is  now  based on the product of the scores of the daughter particles.

Similarly to the linear case, the Poisson-Boltzmann equation~\eqref{eq:nonlinearizedPB} can be reformulated as two PDEs in
$\Omega_{\text{in}}$ and $\Omega_{\text{out}}$, and the transmission condition~\eqref{eq:raccordement}. The PDE in
$\Omega_{\text{in}}$ is still~\eqref{eq:PDE-1},  but in $\Omega_{\text{out}}$ the PDE~\eqref{eq:PDE-2} is replaced by
\begin{equation}
  \label{eq:PB-ext}
    \begin{cases}
  -\frac{1}{2}\Delta u(x)+\lambda_0\sinh u(x)=0,&
    \mbox{for\ }x\in \Omega_{\rm{out}} \\
    u(x)=h(x) & \mbox{for\ }x\in\Gamma,
  \end{cases}  
\end{equation}
where
$
\lambda_0:=\frac{\kappa_{\text{out}}^2}{2} = \frac{\bar{\kappa}^2}{2 \epsout}.
$

Subsection~\ref{sec:QL-PDE--BBM} is devoted to the description of the branching Brownian motion and its link with
elliptic non-linear PDEs. Our algorithms are then described in Subsection~\ref{sec:mainstep}, and numerical
experiments are described in Subsection~\ref{sec:cas-test}.

\subsection{Quasi-linear elliptic PDEs and Branching Brownian motion} \label{sec:QL-PDE--BBM}

\subsubsection{Branching Brownian motion}
\label{sec:BBM}

Let us construct a branching Brownian motion (denoted BBM below) in $\Omega_{\text{out}}$ by defining the dates of birth ($\theta$)
and death ($\sigma$) of each particle and the positions of the particles between their birth and death times. We use the classical
Ulam-Harris-Neveu labelling for individuals, i.e.\ each individual is labeled by an element of ${\cal H}=\bigcup_{n\geq
  0}\mathbb{N}^n$. The ancestor is denoted by $\emptyset$ ($=\mathbb{N}^0$ by convention), and the $j$-th child of an individual
$u\in{\cal H}$ is denoted by $uj$, where $uj$ stands for the concatenation of the vector $u$ and the number $j$. Some of the
particles of ${\cal H}$ never appear in the population, so we need to introduce a cemetery point $\partial$ to code for those
individuals. More precisely, when $v\in{\cal H}$ satisfies $\theta_v=\sigma_v=\partial$, it  means that the individual $v$
never lived in the BBM.

We now introduce the following independent stochastic objects:
\begin{itemize}
\item let $(B^v_t,t\geq 0)_{v\in{\cal H}}$ be i.i.d.\ Brownian motions;
\item let $(E_v)_{v\in{\cal H}}$ i.i.d.\ exponential r.v.\ of parameter $\lambda$;
\item let $(K_v)_{v\in{\cal H}}$ i.i.d.\ r.v.\ in $\mathbb{N}$ with distribution $(p_k)_{k\geq 0}$,
\end{itemize}
the parameters $\lambda$ and $p_k$, $k\geq 0$, will be chosen later.

Fix $x\in\Omega_{\text{out}}$. We construct the BBM started from $x$ (i.e.\ the r.v.\ $\theta_v$, $\sigma_v$ in
$[0,+\infty)\cup\{\partial\}$ and $X^v_t\in\Omega_{\text{out}}$ for $t\in[\theta_v,\sigma_v]$ for all $v\in{\cal H}$) as follows:
\begin{enumerate}
\item The initial particle has position $x\in\Omega_{\text{out}}$; its label is $\emptyset$, $\theta_\emptyset=0$,
  $\sigma_\emptyset=\inf\{t\geq 0:x+B^\emptyset_t\in\Gamma\}\wedge E_{\emptyset}$ and $X^\emptyset_t=x+B^\emptyset_t$ for all
  $t\leq\sigma_\emptyset$.
\item Suppose that the r.v. $\theta_{v}$ (birth time of individual $v$) and $\sigma_v$ (death time of individual $v$) and $X^v_t$,
  $t\in[\theta_v,\sigma_v]$ are constructed
  \begin{enumerate}
  \item If $\sigma_v\not=\partial$ and $X^v_{\sigma_v}\not\in\Gamma$, then for all $1\leq i\leq K_v$, $\theta_{vi}=\sigma_v$, and
    for all $i> K_v$, $\theta_{vi}=\sigma_{vi}=\partial$,  for all $1\leq i\leq K_v$, we also define
\begin{align*}
&\sigma^\Gamma_{vi}=\inf\{t\geq\theta_{vi}:  X^v_{\sigma_v}+B^{vi}_t-B^{vi}_{\theta_{vi}}\in\Gamma\},\\
&\sigma_{vi}=(\theta_{vi}+E_{vi})\wedge \sigma^\Gamma_{vi}, \mbox{and}\\
&X^{vi}_t=X^v_{\sigma_v}+B^{vi}_t-B^{vi}_{\theta_{vi}}
\end{align*}
    for all $t\in[\theta_{vi},\sigma_{vi}]$.
  \item If $\sigma_v\not=\partial$ and $X^v_{\sigma_v}\in\Gamma$, then $\theta_{vi}=\sigma_{vi}=\partial$ for all $i\geq 1$.
  \item If $\sigma_{v}=\partial$, then $\theta_{vi}=\sigma_{vi}=\partial$ for all $i\geq 1$.
  \end{enumerate}
\end{enumerate}
We denote by $\mathbb{P}_x$ the law of the BBM when the first particle has initial position $x$, and $\mathbb{E}_x$ the corresponding
expectation.

Note that, in the case where $\Omega_{\text{out}}=\mathbb{R}^3$, we obtain the standard branching Brownian motion (cf.
e.g.~\cite{mckean-75}), in which the number of particles is a continuous-time branching process $(Z_t,t\geq 0)$ with branching rate
$\lambda$ and offspring distribution $(p_k)_{k\geq 0}$. For general $\Omega_{\text{out}}$, the BBM process can be coupled with the
standard one such that the number of particles alive at time $t$ is smaller than $Z_t$ for all $t\geq 0$.

In particular, in the case where $\sum_{k\geq 0}k p_k\leq 1$, the branching process $Z$ is sub-critical or critical and hence all the
particles of the BBM process either die or reach $\Gamma$ after an almost surely finite time.

\subsubsection{A general probabilistic interpretation of quasi-linear elliptic PDEs in terms of BBM}
\label{sec:PDE-BBM}

The general link between branching Markov processes and non-linear parabolic PDEs is known
since~\cite{skorohod-64,ikeda-nagasawa-al-68a}. The particular case of the binary branching Brownian motion (with $p_2=1$ and $p_k=0$
for all $k\not=2$) has been particularly studied because of its link with the Fisher-Kolmogorov-Petrovskii-Piskunov
PDE~\cite{mckean-75,bramson-78}. This approach has been also used to give a probabilistic interpretation of the Fourier transform of
Navier-Stokes equation~\cite{lejan-sznitman-97}. In contrast, the link between branching diffusions and elliptic PDEs like
Poisson-Boltzmann equation has not been exploited a lot in the literature. Surprisingly, the probabilistic interpretations of parabolic or elliptic non-linear PDEs seem not to have been used a lot for numerical purpose
either (see e.g~\cite{vilelamendes-09,rasulov-raimova-al-10,henry-labordere-12,henry-labordere-al-13}). The general result we
present here about the link between non-linear elliptic PDEs and the BBM, and our method of proof, are original (as far as we know).

Let $g:\mathbb{Z}_+\rightarrow \mathbb{R}$ and $h:\Gamma\rightarrow\mathbb{R}$ be bounded functions and assume that the power series
$\sum_{k\geq 0}p_k g(k) x^k$ has infinite radius of convergence. We consider the following quasi-linear elliptic PDE:
\begin{equation}
  \label{eq:general-PDE}
  -\frac{1}{2}\Delta u (x)+\lambda u(x)-\lambda\sum_{k\geq 0}p_k g(k)u(x)^k=0,\qquad x\in\Omega_{\text{out}},
\end{equation}
with boundary condition $u(x)=h(x)$ for all $x\in\Gamma$.

\begin{thm}
  \label{thm:FK}
  Consider the BBM of the previous subsection and assume that the PDE~\eqref{eq:general-PDE} has a $C^2$ solution $u$, bounded and
  with bounded first-order derivatives. For all $n\geq 0$, let $\tau_n$ be the first time where the BBM has $n$ alive particles in
  the whole elapsed time. Then, for all $n\geq 0$ and $x\in\Omega_{\rm{out}}$,
  \begin{equation}
    \label{eq:prob-int-stopped}
    \mathbb{E}(Y_{\tau_n})=u(x),
  \end{equation}
  where
  \begin{align*}
    Y_t & =\prod_{v\in\mathcal{H}}\left[g(K_v)\mathbbm{1}_{\{\sigma_v\leq t,\ X^v_{\sigma_v}\not\in\Gamma\}}+h(X^v_{\sigma_v})
      \mathbbm{1}_{\{\sigma_v\leq t,\ X^v_{\sigma_v}\in\Gamma\}}+u(X^v_t)\mathbbm{1}_{\{\theta_v\leq t<\sigma_v\}}\right] \\
    & = \prod_{v\in\mathcal{H}}\left[g(K_v)\mathbbm{1}_{\{\sigma_v\leq t,\ X^v_{\sigma_v}\not\in\Gamma\}}+u(X^v_{\sigma_v\wedge t})
      \mathbbm{1}_{\{X^v_{\sigma_v}\in\Gamma\text{\ or\ }\theta_v\leq t<\sigma_v\}}\right]
  \end{align*}
  and $Y_\infty=\lim_{t\rightarrow\infty} Y_t$ is well-defined on the event $\{\tau_n=\infty\}$.

  Assume further that $\sum_{k\geq 0}k p_k\leq 1$. Then, $\mathbb{P}(\tau_n<\infty)\rightarrow 0$ when $n\rightarrow\infty$ and 
  $Y_\infty$ is well-defined a.s. Then, in the case where $\mathbb{E}(|Y_\infty|)<\infty$, we have, for all $x\in\Omega_{\text{out}}$,
  \begin{equation}
    \label{eq:prob-int}
    u(x)=\mathbb{E}_x\left\{\prod_{v\in\mathcal{H}:\sigma_v\not=\partial}\left[\mathbbm{1}_{\{X^v_{\sigma_v}\not\in\Gamma\}}\ g(K_v)+\mathbbm{1}_{\{X^v_{\sigma_v}\in\Gamma\}}\
        h(X^v_{\sigma_v})\right]\right\}.
  \end{equation}
\end{thm}

The last formula extends the classical Feynman-Kac formula to a class of non-linear elliptic PDEs.

Note that the existence of a $C^2$ solution to~\eqref{eq:general-PDE} holds for example if $\Gamma$ and $h$ are
$C^\infty$~\cite{gilbarg-trudinger-01}. In the case of Poisson-Boltzmann PDE, a PDE of the form~\eqref{eq:general-PDE} is coupled with
a PDE in $\Omega_{\text{in}}$. In this case, one can also prove that $u$ is $C^2$ on $\overline{\Omega}_{\text{out}}$ if $\Gamma$ is
$C^\infty$~\cite{champagnat-perrin-talay-14}.

A simple way to get the heuristics behind the probabilistic interpretation is the following. If we assume that the r.h.s.\
of~\eqref{eq:prob-int}, denoted below by $\hat{u}(x)$, is well-defined and smooth, then we can use the following standard technique
for branching processes: distinguish between the different events that may occur between times 0 and $h$, apply the Markov property,
and let $h$ converge to 0. More precisely, let us denote by $Z$ the r.v.\ in the expectation in the r.h.s.\ of~\eqref{eq:prob-int}.
First, we have of course $\hat{u}(x)=h(x)$ for all $x\in\Gamma$. Next, for $x\in\Omega_{\text{out}}$, we can write
\begin{equation*}
  \hat{u}(x) =\mathbb{E}_x\left[\mathbbm{1}_{\{\sigma_\emptyset\leq h\}}\ g(K_\emptyset)\
    \hat{u}(X^\emptyset_{\sigma_\emptyset})^{K_\emptyset}\right] +\mathbb{E}_x\left[\mathbbm{1}_{\{\sigma_\emptyset> h\}}\
    \hat{u}(X^\emptyset_{h})\right].
\end{equation*}
Now, since $x\not\in\Gamma$, the first hitting time $\tau_\Gamma$ of $\Gamma$ by $(x+B^\emptyset_t,t\geq 0)$ satisfies
$\mathbb{P}(\tau_\Gamma\leq h)=o(h)$ at least for smooth $\Gamma$ (it actually decreases exponentially in $1/h$, since $\tau_\Gamma$
is larger than the exit time of a Brownian motion from a fixed ball, the distribution of which is known~\cite {borodin-salminen-02}).
Therefore, except on an event of probability $o(h)$, $\sigma_\emptyset=E_\emptyset$ on the event $\{\sigma_\emptyset\leq h\}$, and
\begin{align*}
  \mathbb{E}_x\left[\mathbbm{1}_{\{\sigma_\emptyset\leq h\}}\ g(K_\emptyset)\
    \hat{u}(X^\emptyset_{\sigma_\emptyset})^{K_\emptyset}\right] & =\mathbb{E}_x\left[\mathbbm{1}_{\{E_\emptyset\leq h\}}\ g(K_\emptyset)\
    \hat{u}(X^\emptyset_{\sigma_\emptyset})^{K_\emptyset}\right]+o(h) \\ & =(1-e^{-\lambda h})\sum_{k\geq 0}p_k g(k) \mathbb{E}[\hat{u}(x+B^\emptyset_{E^{(h)}})^k]+o(h),
\end{align*}
where $E^{(h)}$ is an exponential r.v.\ with parameter $\lambda$, conditioned to be smaller than $h$, independent of $B^\emptyset$.
Note that, in the last equation, we implicitly extended the function $\hat{u}$ as a bounded function on $\mathbb{R}^3$. Under the
assumption that $\hat{u}$ is bounded and continuous, and the power series $\sum_k p_k g(k) x^k$ has an infinite radius of convergence,
it is elementary to prove that
$$
\lim_{h\rightarrow 0}\sum_{k\geq 0}p_k g(k) \mathbb{E}[\hat{u}(x+B_E)^k]=\sum_{k\geq 0}p_k g(k) \hat{u}(x)^k.
$$
Similarly,
\begin{equation*}
  \mathbb{E}_x\left[\mathbbm{1}_{\{\sigma_\emptyset>h\}}\
    \hat{u}(X^\emptyset_{h})\right] =\mathbb{E}_x\left[\mathbbm{1}_{\{E_\emptyset> h\}}\ 
    \hat{u}(x+B^\emptyset_h)\right]+o(h).
\end{equation*}
Since $E^\emptyset$ is independent of $(B^\emptyset_t)$ and $\hat{u}$ has been assumed to be twice continuously differentiable, It\^o's
formula gives
\begin{equation*}
  \mathbb{E}_x\left[\mathbbm{1}_{\{\sigma_\emptyset>h\}}\
    \hat{u}(X^\emptyset_{h})\right] =e^{-\lambda h}\left(\hat{u}(x)+\frac{1}{2}\int_0^h\mathbb{E}[\Delta \hat{u}(x+B^\emptyset_s)]ds\right)+o(h).
\end{equation*}
Again, if $\Delta\hat{u}$ is bounded and continuous, one has
$$
\lim_{h\rightarrow 0}\frac{1}{h}\int_0^h\mathbb{E}[\Delta \hat{u}(x+B^\emptyset_s)]ds=\Delta \hat{u}(x).
$$
Combining all the previous estimates, we obtain
$$
\hat{u}(x)=\lambda h\sum_{k\geq 0}p_k g(k) \hat{u}(x)^k+(1-\lambda h)\left( \hat{u}(x)+h\frac{\Delta \hat{u}(x)}{2}\right)+o(h),
$$
which entails~\eqref{eq:general-PDE} in the limit $h\rightarrow 0$.
\bigskip

The last argument, although intuitive, requires a priori regularity for the function $\hat{u}(x)$. In addition, this method does not
allow to obtain results in cases where the expectation in the right-hand side of~\eqref{eq:prob-int} is not finite. This is why we
prefer to give a new proof extending the classical method of proof of Feynman-Kac formula~\cite{friedman-75} to the case of branching
diffusions. The idea is to compute the semimartingale decomposition of the process $(Y_t,t\geq 0)$.

\begin{proof}[Proof of Theorem~\ref{thm:FK}]
  
It is more convenient to use an equivalent construction of the BBM: let us consider the following independent stochastic
objects:
\begin{itemize}
\item let $(B^v_t,t\geq 0)_{v\in{\cal H}}$ be i.i.d.\ Brownian motions;
\item let $(N^v(dt,dk))_{v\in\mathcal{H}}$ be i.i.d.\ Poisson point measures on $[0,\infty)\times\mathbb{Z}_+$ with intensity measure
  $q(dt,dk)=\lambda\sum_{i\geq 0}p_i \delta_i(dk)dt$, where $\delta_i$ is the Dirac measure at
  the point $i$.
\end{itemize}
The following construction of the BBM amounts to define $E_v$ as the time before the first atom of $N^v$ after $\theta^v$, and $K_v$ as
the second coordinate of this atom. The state of the BBM will be represented as the counting measure
$$
\nu_t=\sum_{v\in\mathcal{H}}\mathbbm{1}_{\{\theta_v\leq t<\sigma_v\}}\delta_{(X^v_t,v)},
$$
constructed as follows: for all bounded function $f:\mathbb{R}^3\times\mathcal{H} \rightarrow \mathbb{R}$, twice continuously
differentiable w.r.t.\ the first variable with uniformly bounded derivatives, the measure $\nu_t$ satisfies
\begin{align*}
  \langle\nu_t,f\rangle &= f(x,\emptyset)+\int_0^t \sum_{v\in\mathcal{H}}\langle\nu_{s-},\nabla f e_v\rangle d
  B^v_s+\frac{1}{2}\int_0^t\sum_{v\in\mathcal{H}}\langle\nu_{s-},\Delta f e_v\rangle ds \\ & +\int_0^t\int_{\mathbb{N}}\sum_{v\in\mathcal{H}}
  \left(\sum_{j=1}^k\langle\nu_{s-},f(\cdot,vj)e_v\rangle-\langle\nu_{s-},fe_v\rangle\right)N^v(ds,dk),
\end{align*}
where $\langle\nu,f\rangle$ stands for $\int_{\mathbb{R}^3\times\mathcal{H}} f(x,v)\nu(dx,dv)$ and
$e_v(x,w)=\mathbbm{1}_{\{x\in\Omega_{\text{out}}\}}\mathbbm{1}_{\{w=v\}}$, so that $\langle\nu_s,e_v\rangle=\mathbbm{1}_{\{\theta_v\leq
  t<\sigma_v\}}$. These equations characterize the measure $\nu_t$ for all $t\geq 0$, and they simply rephrase the algorithmic
construction of the BBM given above.

Introducing the compensated Poisson point measures $\tilde{N}^v(ds,dk)=N^v(ds,dk)-q(ds,dk)$, the last formula immediately gives the
semimartingale decomposition of functionals of $\nu_t$ of the form $\langle\nu_t,f\rangle$:
$$
\langle\nu_t,f\rangle= f(x)+\frac{1}{2}\int_0^t\sum_{v\in\mathcal{H}}\langle\nu_{s-},\Delta f e_v\rangle ds+
\lambda\int_0^t\int_{\mathbb{N}}\sum_{v\in\mathcal{H}} \left(\sum_{k\geq
    0}p_k\sum_{j=1}^k\langle\nu_{s-},f(\cdot,vj)e_v\rangle-\langle\nu_{s-},fe_v\rangle\right)ds +M^f_t,
$$
where $M^f_t$ is the local martingale
$$
M^f_t=\int_0^t \sum_{v\in\mathcal{H}}\langle\nu_{s-},\nabla f e_v\rangle d
  B^v_s+\int_0^t\sum_{v\in\mathcal{H}}
  \left(\sum_{j=1}^k\langle\nu_{s-},f(\cdot,vj)e_v\rangle-\langle\nu_{s-},fe_v\rangle\right)\tilde{N}^v(ds,dk).
$$

Of course, the semimartingale decomposition of other functionals of $(\nu_t,t\in[0,T])$ can be obtained in a similar way. 
Given $x\in\Omega_{\text{out}}$, under $\mathbb{P}_x$, we obtain for the process $Y_t$
\begin{align*}
  Y_t & =u(x)+\int_0^t Y_{s-}\sum_{v\in{\cal H}}\langle\nu_{s-},\frac{\nabla u}{u}e_v\rangle d
  B^v_s+\frac{1}{2}\int_0^t Y_{s-}\sum_{v\in{\cal H}}\langle\nu_{s-},\frac{\Delta
    u}{u}e_v\rangle ds \\ & +\int_0^t\int_{\mathbb{N}} Y_{s-}\sum_{v\in{\cal
      H}}\left(\langle\nu_{s-},u^{k-1}g(k)e_v\rangle-\langle
    \nu_{s-},e_v\rangle\right)N^v(ds,dk) \\ & =u(x)+\int_0^t Y_{s-}\sum_{v\in{\cal H}}
  \langle\nu_{s-},\frac{e_v}{u}\left(\frac{1}{2}\Delta u+\lambda\sum_{k\geq 0}p_k g(k)
    u^{k-1}-\lambda u\right)\rangle ds+M_t \\ & =u(x)+M_t,
\end{align*}
where
$$
M_t=\int_0^t Y_{s-}\sum_{v\in{\cal H}}\langle\nu_{s-},\frac{\nabla u}{u}e_v\rangle d
  B^v_s+\int_0^t\int_{\mathbb{N}} Y_{s-}\sum_{v\in{\cal
      H}}\left(\langle\nu_{s-},u^{k-1}g(k)e_v\rangle-\langle
    \nu_{s-},e_v\rangle\right)\tilde{N}^v(ds,dk)
$$
is a local martingale. Note that in the previous computation, we made the convention that, for
any $t\geq 0$, $v\in{\cal H}$ and $x$ s.t.\ $\nu_t(\{(x,v)\})=1$,
$$
Y_t/u(x)=\prod_{w\in\mathcal{H},\ w\not=v}\left[g(K_w)\mathbbm{1}_{\{\sigma_w\leq t,\ X^w_{\sigma_w}\not\in\Gamma\}}+u(X^w_{\sigma_w\wedge t})
\mathbbm{1}_{\{X^w_{\sigma_w}\in\Gamma\text{\ or\ }\theta_w\leq t<\sigma_w\}}\right].
$$
Since $u$ and $\nabla u$ are bounded functions, we have of course that
$(M_{t\wedge\tau_n},t\geq 0)$ are martingales for all $n\geq 0$, and so
$$
\mathbb{E}(Y_{t\wedge\tau_n})=u(x)
$$
for all $t\geq 0$ and $n\geq 1$. Since $(Y_{t\wedge\tau_n},t\geq 0)$ is uniformly bounded, we
obtain~\eqref{eq:prob-int-stopped}.

In the case where $\sum_{k\geq 0}kp_k\leq 1$, the number of particles in the BBM is smaller than the number of particles in a
sub-critical or critical continuous-time branching process defined as $Z_t=\langle\mu_t,\mathbbm{1}\rangle$, where
$$
\mu_t=\delta_\emptyset+\int_0^t\int_{\mathbb{N}}\sum_{v\in{\cal
    H}}\left(\sum_{j=1}^k\delta_{vj}-\delta_v\right)\mu_{s-}(\{v\})N^v(ds,dk).
$$
Then of course $\mathbb{P}(\tau_n<\infty)\rightarrow 0$ when $n\rightarrow 0$, $Y_\infty$ is
a.s.\ well-defined, and~\eqref{eq:prob-int} is clear if
$\mathbb{E}|Y_\infty|<\infty$.
\end{proof}

\subsubsection{A probabilistic interpretation of the nonlinear Poisson-Boltzmann equation}
\label{sec:first-choice}

The last theorem gives a probabilistic interpretation of PDEs of the form of Poisson-Boltzmann equation (outside the molecule), which
suggests to use again a Monte-Carlo method to estimate $u(x)$.
The difficulty is to find the good probabilistic interpretation allowing to use~(\ref{eq:prob-int}) rather
than~(\ref{eq:prob-int-stopped}) which involves the unknown function.

One possible choice of the function $g$ and the probability distribution $(p_k)_{k\geq 0}$ to recover the PDE~\eqref{eq:PB-ext} is as
follows:
\begin{equation}  \label{eq:choix-param}
\left\{
  \begin{aligned}
    &\lambda=\lambda_0; \\ 
    &g(k)=-1 \text{ for all } k\geq 1,\ g(0)=0; \\
    &p_{2k+1}=\frac{1}{(2k+1)!} \text{ and } p_{2k}=0 \text{ for all } k\geq 1,\ p_1=0; \\
    &p_0=1-\sum_{k\geq 1}p_{2k+1}=2-\sinh 1>0.
  \end{aligned}
  \right.
\end{equation}
With these parameters, we have $\displaystyle
\sum_{k\geq 0}k p_k=\cosh 1-1<1$, so the BBM goes extinct after an a.s.\ finite time.

In this case, Theorem~\ref{thm:FK} gives the following probabilistic interpretation for the Poisson-Boltzmann PDE.
\begin{cor}
  Let $u$ be the solution of the PDE~\eqref{eq:PB-ext}: provided that the expectation is
  well-defined, for all $x\in\Omega_{\rm{out}}$,
  \begin{equation}
    \label{eq:interpr-proba-ext}
    u(x)=\mathbb{E}_x\left\{(-1)^{\#\{v:X^v_{\sigma_v}\not\in\Gamma\}}\prod_{v:\sigma_v\not=\partial}\left[\mathbbm{1}_{\{X^v_{\sigma_v}\not\in\Gamma\}}\
          \mathbbm{1}_{\{K_v\geq 1\}}+\mathbbm{1}_{\{X^v_{\sigma_v}\in\Gamma\}}\ u(X^v_{\sigma_v})\right]\right\},
  \end{equation}
  where the BBM $(X^v_t,t\geq 0, v\in{\cal H})$ has parameters~\eqref{eq:choix-param}.
\end{cor}

In view of the last formula, it is clear that the expectation in the right-hand side is finite if $u$ is bounded by 1 on $\Gamma$.
When $u$ takes larger values on $\Gamma$, the product involved in~\eqref{eq:interpr-proba-ext} can be bounded by some exponential
moment of the total number of particles which hit $\Gamma$ in the BBM, but this bound is not finite in general. Therefore, we cannot
ensure in general that the expectation is well-defined with so simple estimates. Note that the random variable inside the expectation
in~\eqref{eq:interpr-proba-ext} is zero when one of the particles in the BBM dies before hitting $\Gamma$ and gives birth to no
children, so one can expect better bounds on the expectation, but such bounds seem very delicate to obtain. We will not study this
question here.

\begin{rk}
  \label{rem:variance}
  In this work, we concentrate on the characterization of the solution of Poisson-Boltzmann equation as the expectation of a random
  variable $Z$ defined from a BBM as in~\eqref{eq:interpr-proba-ext}. We do not provide a complete study of the variance of $Z$,
  which requires a deeper analysis. One can give an idea of the possible difficulties with the previous tools as follows:
  because $Z$ is defined as a product, it can be seen from Theorem~\ref{thm:FK}, using the parameter values~\eqref{eq:choix-param},
  that $v(x)=\mathbb{E}_x(Z^2)$ should be (at least formally) solution to the PDE
  \begin{equation}
    \label{eq:EDP-variance}
    \begin{cases}
      -\frac{1}{2}\Delta v(x)+3\lambda v(x)-\lambda\sinh v(x)=0, \\
      v_{\mid\Gamma}=u^2_{\mid\Gamma}.
    \end{cases}
  \end{equation}
  This elliptic PDE has a non-linear term which is concave where the classical theory would require it to be convex (as in
  Poisson-Boltzmann equation). In particular, this prevents to use classical convexity arguments to characterize the solution of this
  PDE as the solution of a variational problem (see for example~\cite{champagnat-perrin-talay-14} for the application of these
  arguments for Poisson-Boltzmann PDE). This problem will be crucial if the solution takes large values, typically when the boundary
  condition is too large, because then one cannot expect a good approximation of $v$ by the linearization of~\eqref{eq:EDP-variance},
  and the classical arguments cannot ensure that a solution to this PDE exists.
\end{rk}

\subsection{The main steps of the Monte-Carlo algorithm}\label{sec:mainstep}

\subsubsection{Outside the molecule: branching walk on spheres (BWOS) algorithm}
\label{sec:algo}

The walk on spheres (WOS) algorithm of Section~\ref{sec:outside} can be extended in order to deal with possible branching. The
only difference is that the precise location of death of the particle must be simulated to obtain the initial position  of its daughters.

Consider $y\in\Omega_{\text{out}}$ and a radius $R$ such that $B(y,R)\subset\Omega_{\text{out}}$. Consider also a Brownian motion
$(B_t,t\geq 0)$ such that $B_0=y$ and let $\tau_R$ be the first hitting time by $B_t$ of the sphere $\partial B(y,R)$. A WOS
algorithm sampling the system of branching Brownian particles requires to simulate the position of $B_{\tau_R\wedge E}$, where $E$ is
an independent exponential random variable of parameter $\lambda$. Because of the spherical symmetry of the Brownian path, the only
relevant information is the p.d.f.\ of $R_{\tau_R\wedge E}$, where $R_t=|B_t-y|$ is a Bessel process of dimension 3. This
p.d.f.\ can be computed from the formula~\cite[Ch.\:5,\ Formula\:1.1.6]{borodin-salminen-02}
$$
\PP_0\left(\sup_{0<s<E}R_s<R,\ R_E\in dr\right)=\frac{2\lambda r
  \sinh\left[(R-r)\sqrt{2\lambda}\right]}{\sinh\left(R\sqrt{2\lambda}\right)},\quad\forall
0\leq r\leq R,
$$
where, under $\mathbb{P}_0$, $B$ is a standard Brownian motion started from $0$. From this, we obtain for all $0\leq r\leq R$
$$
\PP_0\left(\sup_{0<s<E}R_s<R,\ R_E\leq r\right)=
1-\frac{r\sqrt{2\lambda}\cosh\left[(R-r)\sqrt{2\lambda}\right]+ 
\sinh\left[(R-r)\sqrt{2\lambda}\right]}{\sinh\left(R\sqrt{2\lambda}\right)}.
$$
Note that, by taking $r=R$, we recover the death probability given in the WOS algorithm of Section~\ref{sec:outside}, as expected. 
Hence, we obtain the explicit cumulative distribution function of $R_E$ conditionally on $\{E<\tau_R\}$
\begin{equation}
  \label{eq:cdf-BWOS}
  F(r)=\mathbb{P}\left(R_E\leq r \, | \, \sup_{0<s<R}R_s<R \right)=\frac{\sinh[r\sqrt{2 \lambda}] -r\sqrt{2\lambda}\cosh[(R-r)\sqrt{2\lambda}]- 
    \sinh[(R-r)\sqrt{2\lambda}] }{\sinh[R\sqrt{2 \lambda}] -R \sqrt{2 \lambda}}.
\end{equation}
One can sample from this distribution by several ways, among which we tested acceptance-rejection methods with several proposition
distributions (among which the Beta$(2,2)$ distribution), and Newton's algorithm to invert the cumulative distribution function. It
appears that Newton's algorithm gives a precise sampling and is much faster than acceptance-rejection algorithms, so we use this
method in our code. More precisely, we will denote by $split(y,R)$ the random variable uniformly distributed on the sphere of center
$y$ and radius $r$ obtained by applying Newton's method with 4 iterations to approximate $F^{-1}(U)$, where $U$ is a uniform
r.v.\ on $[0,1]$ and $F$ is given by~\eqref{eq:cdf-BWOS}.

The BWOS algorithm started at $x\in \Omega_{\text{out}}$ returns a random position which belongs to $\Gamma\cup\Omega_{\text{out}}$.
This random variable is an approximation of $B_{\tau\wedge E}$, where $B$ is a standard Brownian motion started from $x$, $\tau$ its
first hitting time of $\Gamma$, and $E$ an independent exponential r.v.\ of parameter $\lambda$.
\medskip

{\tt
\begin{enumerate}[noitemsep]
\item[]{\bf BWOS Algorithm}
\item[] Given $x\in \Omega_{\text{out}}$, $\lambda\geq 0$, and
$\varepsilon >0$
\item Use the WOS algorithm.
\item IF the particule is not alive, THEN
 \begin{enumerate}[noitemsep]
 \item Denote $y$ the last known position of the particle before the death, given by the WOS algorithm,
 \item let $S(y,r)$ be the largest open sphere included in $D$ centered at $y$, 
 \item return an independent copy of $split(y)$.
 \end{enumerate}
\item ELSE return $exit(x)$ simulated by WOS algorithm.
\item[] END. 
\end{enumerate}
}
\medskip

We denote by $split$\_$or$\_$exit(x)$ the position returned by this algorithm.

\subsubsection{Branching Algorithm}
\label{sec:algo2}

The idea of our algorithm is to use the double randomization technique, as in the linear case and as in~\cite{mascagni-simonov-04}, that  consists to use the approximations~\eqref{eq:FK-inside},~\eqref{eq:jump} and~\eqref{eq:interpr-proba-ext} recursively to estimate
the unknown function $u$ appearing on the right-hand side of each of these formulas. 

Therefore, the method consists when the initial particle starts in $\Omega_{\text{in}}$ to use the UWOS algorithm to simulate its
first hitting point of $\Gamma$, or when the initial particle starts in $\Omega_{\text{out}}$ the WOS algorithm, then use a jump
method as in Section~\ref{sec:boundary} to move the particle away from $\Gamma$, and continue to use inductively the UWOS and WOS
algorithm depending whether the particle entered inside the molecule or exited outside the molecule. This part is exactly the same
as in the linear case, and stops when the particle is killed outside the molecule. This time of killing is actually a branching time,
since a random number of new particles, with distribution $(p_k)_{k\geq 0}$, appears at the death position of the initial particle,
and each new particle continues to evolve independently as the first one.

As in the linear case, we can use the \texttt{SNJ} and \texttt{ANJ} jump methods to move a particle after it reaches $\Gamma$.
Therefore, we set $p(x)=p_{\text{SNJ}}(x)$~\eqref{eq:def-p-SNJ} or $p_{\text{ANJ}}(x)$~\eqref{eq:def-p-ANJ} depending on the method
chosen. Because of the non-linearity, as explained in Section~\ref{sec:boundary}, one cannot expect a better precision for the
\texttt{TAJ} method, which is therefore not used here.

One of the difficulties of the algorithm consists in the computation of the global score of the algorithm. Before the first particle is
killed, the score is computed exactly as in the linear case. Then, if this particle has daughters, because
of~\eqref{eq:interpr-proba-ext}, we need to add to the score of the first particle minus the product of the scores of its daughters
(the minus comes from the term $(-1)^{\#\{v:X^v_{\sigma_v}\not\in\Gamma\}}$ in the r.h.s. of~\eqref{eq:interpr-proba-ext}). These
scores might be 0 if the daughter dies before hitting $\Gamma$ and gives birth to no children, or might be different from zero if the
particle hits $\Gamma$. Indeed, in this case, the particle first jumps according to~\eqref{eq:jump} and then scores some value if it
jumps inside the molecule, according to~\eqref{eq:FK-inside}. But its score might also involve the scores of its own daughters if it
is non-zero. Therefore, the convenient way to formulate our algorithm is by a recursive procedure.

Let $M$ be a r.v. distributed according to the offspring distribution of the BBM:
\begin{equation}
  \label{eq:def-K}
  M=
  \begin{cases}
    0 & \mbox{with probability\ }2-\sinh 1 \\
    2k+1 & \mbox{with probability\ }\displaystyle{\frac{1}{(2k+1)!}},\mbox{\ for all\ }k\geq 1.   
  \end{cases}
\end{equation}
Given $x_0\not\in\{c_1,\ldots,c_N\}$, our branching algorithm can be described recursively as follows.

\medskip

{\tt
\begin{enumerate}[noitemsep]
\item[]{\bf BA$(x_0)$ algorithm.}
\item[]Set $k=0$ and $score=u_0(x_0)$ if $x_0\in\Omega_{\text{in}}$ or $score=0$ otherwise
\item IF $x_k\in\Omega_{\text{in}}$
  \begin{enumerate}[noitemsep]
  \item THEN
    \begin{enumerate}
    \item Use the UWOS algorithm to simulate $exit(x_k)$
    \item Set $score =score - u_0(exit(x_k))$
    \item Set $x_{k+1}$ equal to an independent copy of $p(exit(x_k))$
    \end{enumerate}
  \item ELSE
    \begin{enumerate}
    \item Use the BWOS algorithm with $\lambda=\bar{\kappa}^2/2\varepsilon_{\text{out}}$ to simulate $split$\_$or$\_$exit(x_k)$
    \item IF $split$\_$or$\_$exit(x_k)\in\Omega_{\text{out}}$
    \item THEN
      \begin{enumerate}
      \item Sample $m$ as an independent copy of $M$.
      \item Let $score_1,\ldots,score_m$ be the scores returned by $m$ independent runs of the BA$(split$\_$or$\_$exit(x_k))$ algorithm.
      \item Return $score-\prod_{j=1}^m score_j$ if $m\not=0$, or $score$ if $m=0$.
      \item Goto END.
      \end{enumerate}
    \item ELSE set $x_{k+1}$ equal to an independent copy of $p(split$\_$or$\_$exit(x_k))$.
    \end{enumerate}
  \end{enumerate}
\item IF $x_{k+1}\in\Omega_{\text{in}}$, THEN set $score =score +u_0(x_{k+1})$.
\item Set $k=k+1$ and return to Step~(1).
\item[]END.
\end{enumerate}
}
\medskip

As in the linear case, the approximation of $u(x_0)$ is obtained by computing the Monte-Carlo average of the score returned by the
BA$(x_0)$ algorithm.

The convergence of our algorithm could be analyzed following exactly the same lines as Thm. 4.7 of~\cite{bossy-champagnat-al-09},
giving an error of the same order as the SNJ and ANJ algorithms of Section~\ref{sec:boundary}, provided one could ensure that all the
expectations involved in this algorithm are finite. Contrary to the linear case, because of the products involved in Substep~(C)
above, this is not obvious at all, even if the function $u$ is bounded by 1, so we leave aside this question for a future work.

\subsection{Numerical experiments}
\label{sec:cas-test}

\subsubsection{Implementation}

As in the linear case, the localization of the closest atom from the particle is done using CGAL library, and the parallelization of
the code is done using MPI and SPRGN libraries.

Note that, since the number of children of all particles are i.i.d. and independent of their positions of death, the set of particles
simulated in our algorithm and their genealogical relations are the same as those of a Galton-Watson process with reproduction
distribution $(p_k)_{k\geq 0}$, the distribution of the r.v.\ $M$ of~\eqref{eq:def-K}.

Hence, for the practical implementation of the algorithm, we started by sampling first a Galton-Watson tree with offspring
distribution $(p_k)_{k\geq 0}$ conditioned to be smaller than 10 (the probability to have 11 children or more is smaller than
$5\cdot 10^{-8}$). Then, we simulate the trajectory of each particle as above and we follow the genealogical structure of the tree
previously sampled when a particle is killed.

To avoid the use of recursive functions, which can be time consuming, the simulation of the trajectories of each particle and the
scores obtained along each of these trajectories is done first by exploring the Galton-Watson tree forward in time (from the root to
the leaves), whereas the computation of the total score is done in the end by exploring the tree backward in time (from the leaves to the
root).

Since after a branching the score is evaluated as a product of scores of daughters, it is important to detect fast when one of these
scores is zero, which happens for example when one of the daughters dies without children before hitting $\Gamma$. Hence, when
exploring the Galton-Watson tree forward in time, at each birth event, it is more advantageous in terms of computational time to deal
first with daughters particles that have no children.

The fact that the genealogical tree of our particles is sampled before the simulation of the particles' motion allows us to easily
implement a stratification Monte-Carlo algorithm to reduce the variance of the method. Each stratum is a given Galton-Watson
genealogical tree. Its probability of occurrence is easy to compute. We restricted in our simulation to trees of height 2 or less,
because the probability that the Galton-Watson tree has depth 3 or more is less than $1.3\cdot 10^{-4}$. We run 500 trajectories for
each stratum to estimate the variance within each stratum and allocate the number of runs in a given stratum proportionally to the
probability of the stratum times the empirical standard deviation within the stratum.

\subsubsection{The single atom case}

To evaluate the convergence and the efficiency of the previous algorithm, we made some simulations with a monoatomic molecule. Indeed an approximation of the exact solution of the non-linear Poisson-Boltzmann equation~\eqref{eq:nonlinearizedPB} can be
easily computed in this case. Assuming the unique atom is centered at $0$, the solution $u(x)$ to the Poisson-Boltzmann equation
can be written as $v(\|x\|)$ for some function $v$ because of the
spherical symmetry of the problem. Now, $u-u_0$ is harmonic in $B(0,r)$, with a constant Dirichlet boundary condition thanks again to
the spherical symmetry of the problem. Hence $u-u_0$ is constant in $B(0,r)$, so that $v-v_0$ is constant in $[0,r]$, where
$v_0(\|x\|)=u_0(x)$ for all $x\in\mathbb{R}^3$. Using a spherical coordinates  change of variables, the transmission
problem~\eqref{eq:PDE-1},~\eqref{eq:raccordement} and~(\ref{eq:PB-ext}) can be written as
\begin{equation}
  \label{eq:reference-solution-nonlinear}
  v''(x)+\frac{2}{x}v'(x)=\kappa_{\text{out}}^2\sinh v(x)=0,\quad \text{for all }x\in(r,+\infty),  
\end{equation}
with boundary conditions $\varepsilon_{\text{out}}v'(r)=\varepsilon_{\text{in}}v'_0(r)$ and $v(x)\rightarrow 0$ when
$x\rightarrow+\infty$. We approximate this function as the solution of the same differential equation on $[r,R]$ for a large $R$ with
$u(R)=0$.

The numerical results presented here are obtained using the branching algorithm with jump methods \texttt{SNJ} and \texttt{ANJ}
($\alpha=3$ and $10$), taking the same number of Monte-Carlo simulations ($1\cdot 10^5$) for each run, with a value of $h$ varying
from $0.003$ to $0.9$, and with $\epsilon=10^{-5}$. The reference values are computed with $h=0.001$ and $1\cdot 10^6$ Monte-Carlo simulations. We take a radius $r=1\usk\angstrom$ for the atom centered at zero, and we compute the
value of $(u-u_0)(0)$ using our method and compare it to the value $v(r)-v_0(r)$ computed above.

We have tested different values of the charge $z$, and it appears that the variance of the algorithm is very sensitive to this
parameter. As explained above, this can be understood from formula~(\ref{eq:interpr-proba-ext}) when $|v(r)|>1$, since then the value
computed is similar to an exponential moment of the number of individuals in a sub-critical Galton-Watson tree. In practice, above
some threshold on the parameter $z$, the convergence of the algorithm is much slower. A finer numerical analysis reveals that this is
due to very unprobable Galton-Watson trees, with several large numbers of offsprings, which contribute for an important part to the
empirical variance. Because of the large value of the constant in front of the Dirac masses in~\eqref{eq:terme-source}, the threshold
above which the variance starts to increase drastically is slightly larger than $z=1$.

\begin{figure}[ht]
  \centering
  \subfigure[Monte-Carlo average of the 3 jump methods as a function of $h$ (Log
  scale).\label{fig:average_nonlin_1-atom-z=0.2}]{\includegraphics[width=.45\textwidth]{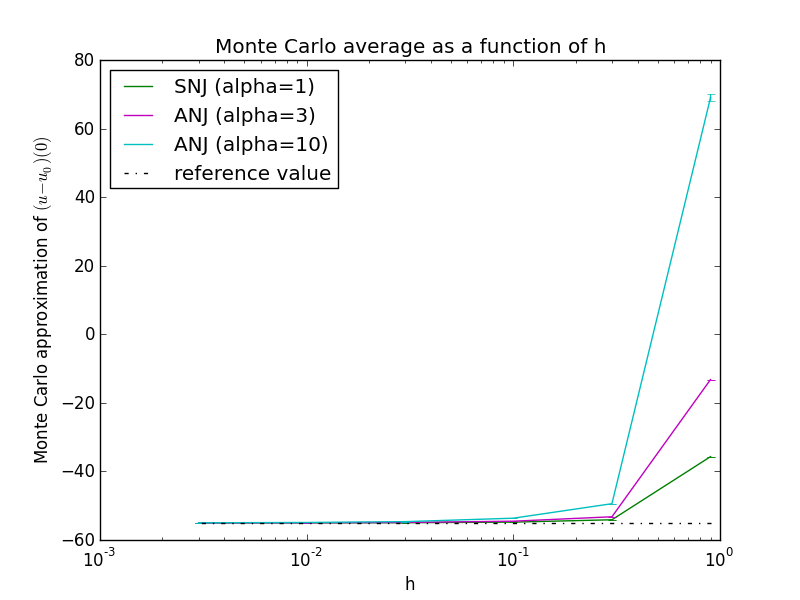}}\quad
  \subfigure[Error (Log scale) of the 3 jump methods w.r.t. the reference value given
  by~\eqref{eq:reference-solution-nonlinear} as a function of $h$ (Log scale).\label{fig:error-nonlin-1-atom-z=0.2}]{
    \includegraphics[width=.45\textwidth]{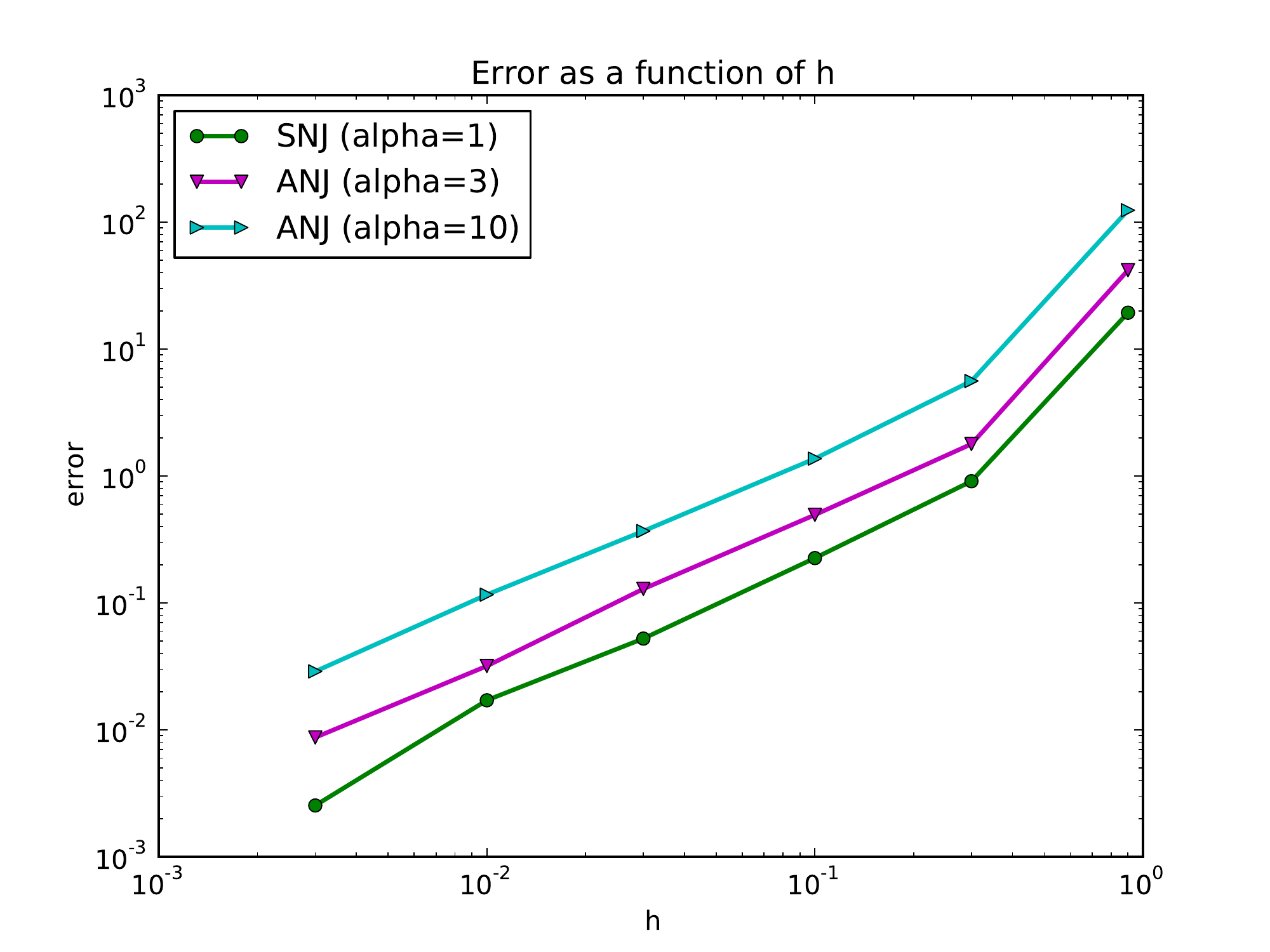}}
  \caption{Convergence and error of the branching algorithm with jump methods \texttt{SNJ} and \texttt{ANJ} ($\alpha=3$ and
    $\alpha=10$) in the case of a single atom with charge $z=0.2$.}\label{fig:nonlin-1-atom-z=0.2}
\end{figure}

This is an important issue of our algorithm, which needs some new ideas to be solved (see the perspectives of Section~\ref{sec:ccl}).
However, for small enough values of $z$, the algorithm behaves very well. We present the numerical results for a much smaller value
$z=0.2$ in Figure~\ref{fig:nonlin-1-atom-z=0.2}, which show a good performance of the algorithm.
Figure~\ref{fig:error-nonlin-1-atom-z=0.2} confirms our conjecture that the error is of the order of $h$. The performance plots in
Figure~\ref{fig:perform-1-atom-z=0.2} indicate that the value of $\alpha$ has a negligible influence on the error of the algorithm for
a given CPU time. Note also that the confidence intervals shown in Figure~\ref{fig:average_nonlin_1-atom-z=0.2} are very small, despite
the relatively small number of Monte-Carlo simulation we used ($3\cdot 10^4$).

\begin{figure}[ht]
  \centering
  \includegraphics[width=.45\textwidth]{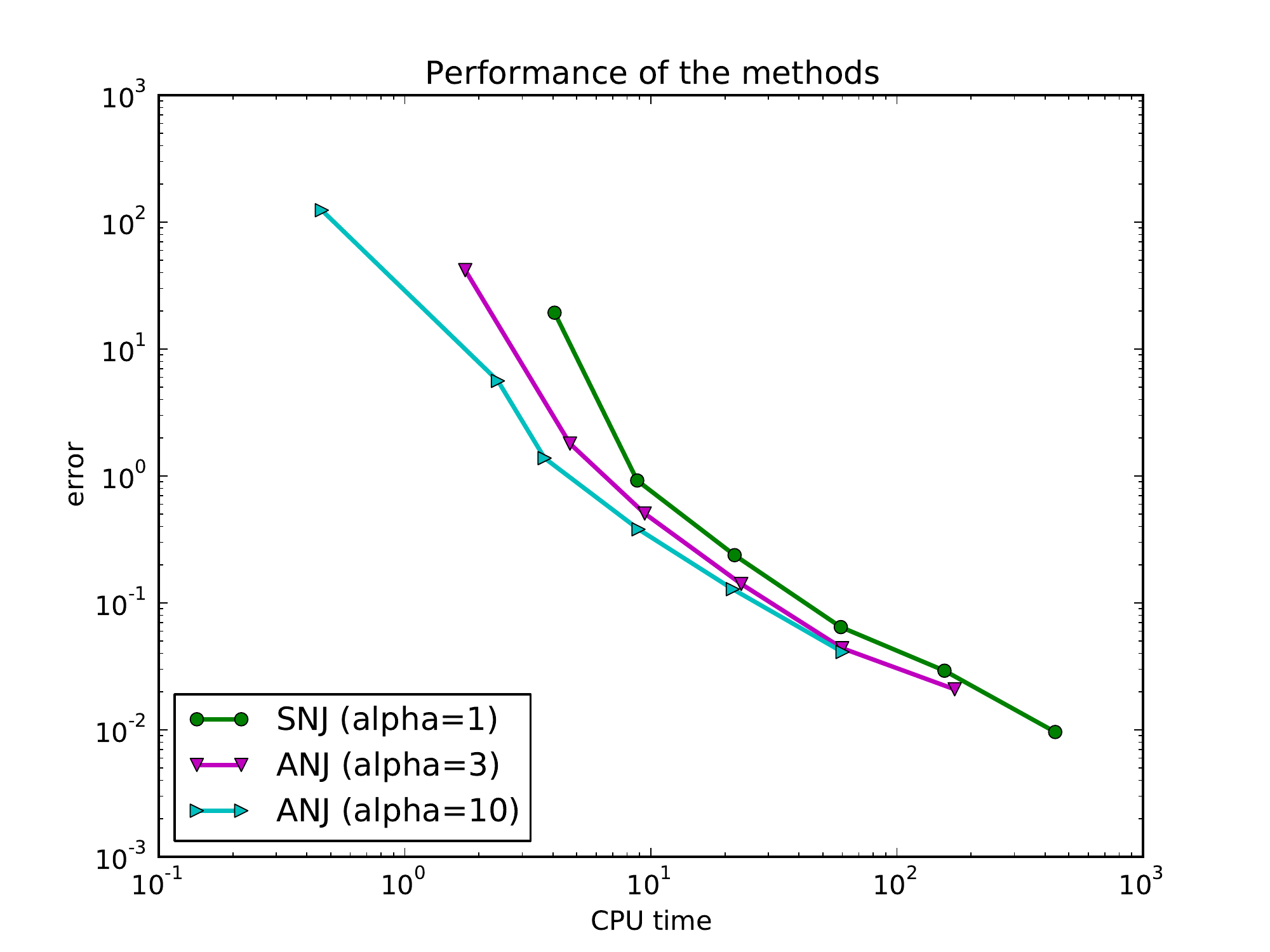}
  \caption{Performance of the branching algorithm with 3 jump methods for a single atom with $z=0.2$: error (Log scale) as a function
    of CPU time (Log scale).}
  \label{fig:perform-1-atom-z=0.2}
\end{figure}

The value $z=1$ (which corresponds to a monoatomic ion of valence $\pm 1$) also has a satisfactory behavior, shown in
Fig~\ref{fig:nonlin-1-atom-z=1}, although the convergence is not as good as for $z=0.2$, as appears in
Figure~\ref{fig:average_nonlin_1-atom-z=1} for $h=0.003$ and in Figure~\ref{fig:error-nonlin-1-atom-z=1}. We also observe a very unstable
behavior of the ANJ method with $\alpha=3$ for large values of $h$. Still, for all values of $h$ between $0.003$ and $0.1$, the
relative error of the algorithm is less than $1$\%. Note that the same simulation, run without stratified Monte-Carlo, gives a very
large variance of the result. According to our tests, the stratification method allowed to increase the threshold of variance
explosion from roughly $z=0.3$ to more than $z=1$.

\begin{figure}[ht]
  \centering
  \subfigure[Monte-Carlo average of the 3 jump methods as a function of $h$ (Log
  scale).\label{fig:average_nonlin_1-atom-z=1}]{\includegraphics[width=.45\textwidth]{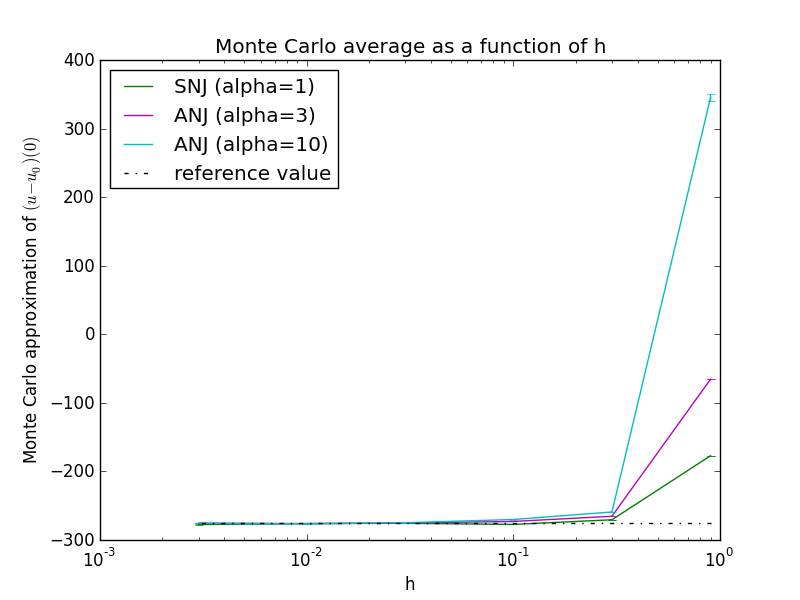}}\quad
  \subfigure[Error (Log scale) of the 3 jump methods w.r.t. the reference value given
  by~\eqref{eq:reference-solution-nonlinear} as a function of $h$ (Log scale).\label{fig:error-nonlin-1-atom-z=1}]{
    \includegraphics[width=.45\textwidth]{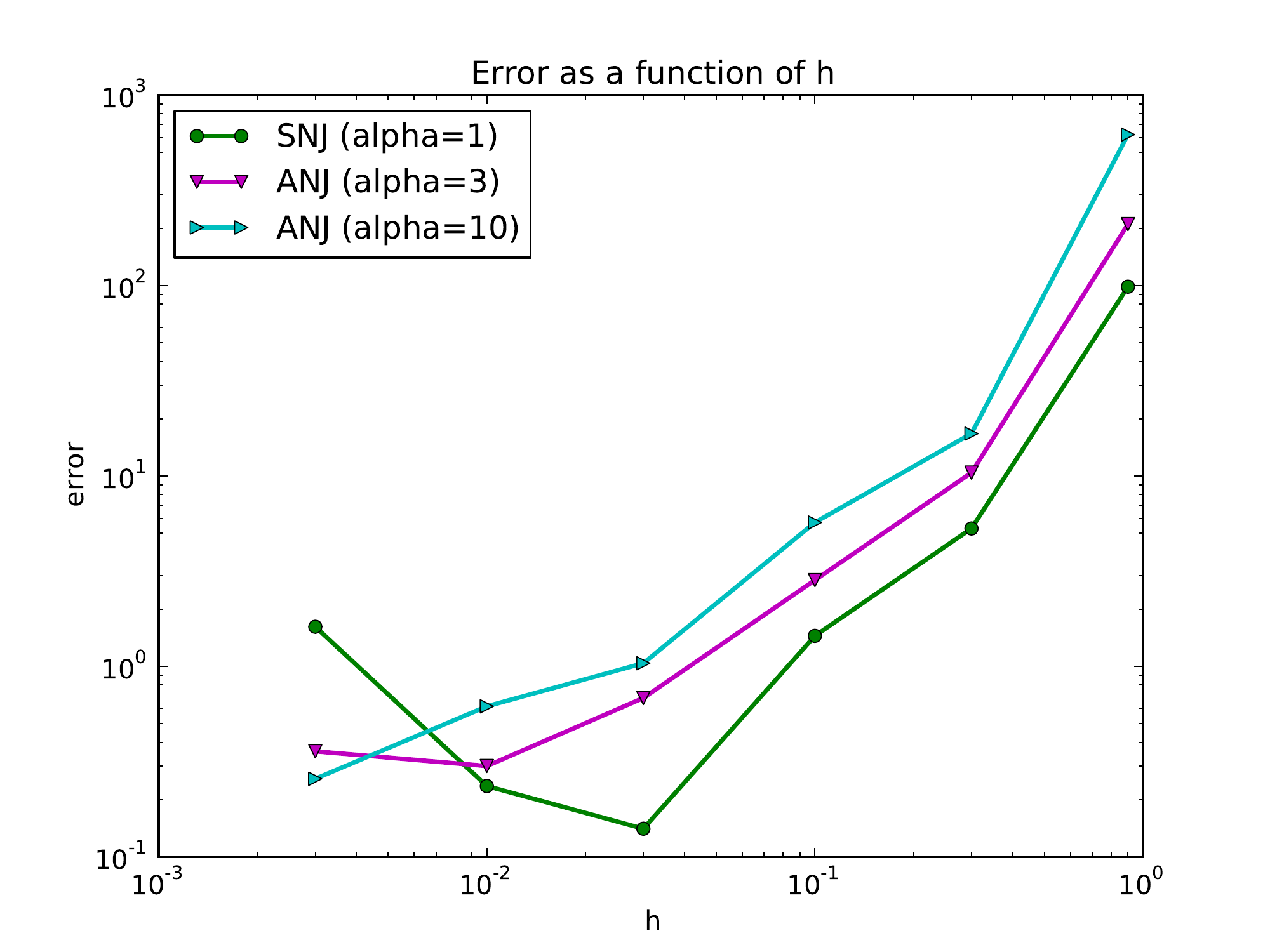}}
  \caption{Convergence and error of the branching algorithm with jump methods \texttt{SNJ} and \texttt{ANJ} ($\alpha=3$ and
    $\alpha=10$) in the case of a single atom with charge $z=1$.}\label{fig:nonlin-1-atom-z=1}
\end{figure}

If one takes a larger value for $z$, the variance starts increasing drastically and the method no longer converges for $h$ small. A
finer analysis of the results of the algorithm shows that the variance is extremely large in a few very unlikely strata,
corresponding to genealogical trees with many children at each generation. Some other runs show a very small empirical variance,
because the very unlikely trajectories with a very large score did not occur by chance.

\subsubsection{The case with two atoms close with opposite charges}

It is generally accepted (cf.\ e.g.~\cite{baker-bashford-al-05}) that for uncharged molecules, the approximation of the non-linear
Poisson-Boltzmann equation by the linear one is not too bad, meaning that the potential $u$ does not take too large values
in $\Omega_{\text{out}}$, and in particular close to $\Gamma$. Since the explosion of the variance observed in the last test case for
large values of $z$ seems to be related to the fact that $u$ takes too large values on $\Gamma$, this suggests to look at uncharged
molecules. The simplest uncharged molecules with non-trivial electrostatic potential $u$ are composed of $N=2$ atoms with opposite
charges. We focus here on this situation, assuming that the two atoms have same radius $r_1=r_2=1\usk\angstrom$ and opposite charges
$z_1=1$ and $z_2=-1$. We denote by $a$ the distance between the centers of the two atoms (in $\angstrom$).

As in the first case, assuming a too large value of $a$ increases the values of $u$ on $\Gamma$ and hence produce variance explosion
of the method. 

\begin{figure}[ht]
  \centering
  \subfigure[Monte-Carlo average for the 3 jump methods as a function of $h$ (Log
  scale).\label{fig:average_nonlin-2-atoms-a=0.2}]{\includegraphics[width=.45\textwidth]{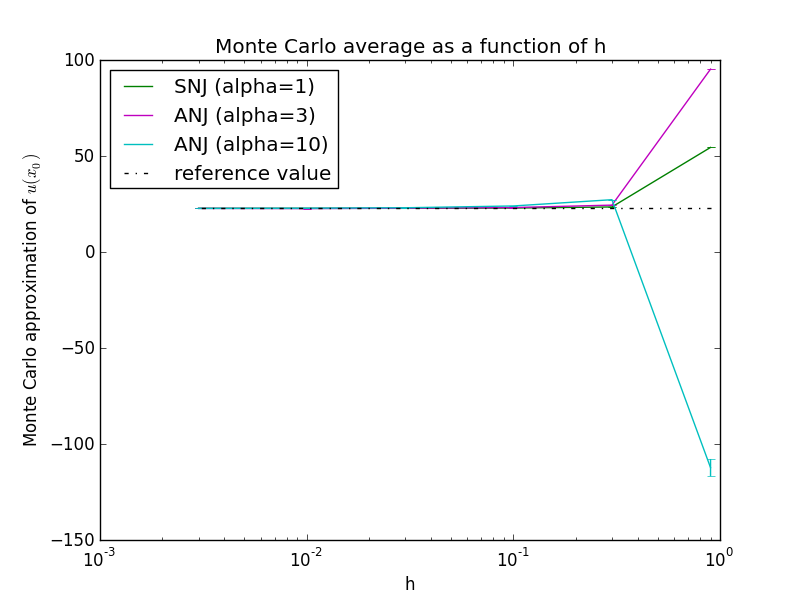}}\quad
  \subfigure[Error (Log scale) of the 3 jump methods as a function of $h$ (Log scale).\label{fig:error-nonlin-2-atoms-a=0.2}]{
    \includegraphics[width=.45\textwidth]{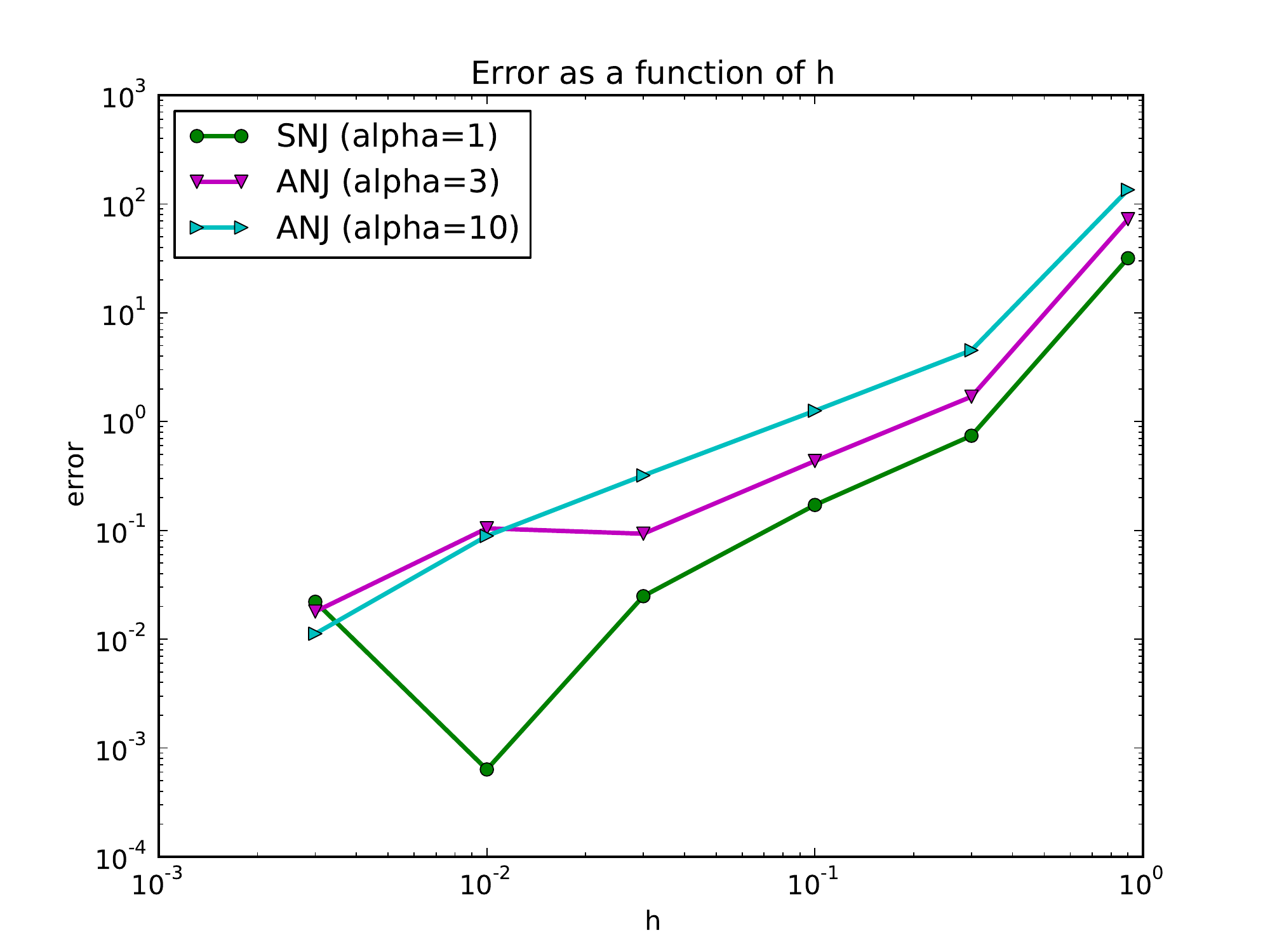}}
  \caption{Convergence and error of the branching algorithm with jump methods \texttt{SNJ} and \texttt{ANJ} ($\alpha=3$ and
    $\alpha=10$) in the case of two atoms with distance $a=0.2$.}\label{fig:nonlin-2-atoms-a=0.2}
\end{figure}

We run similar tests as in the previous example, with $10^5$ Monte-Carlo simulations for several values of $h$ between $0.003$
and $0.9$. We take $\epsilon=10^{-5}$ and we run our branching algorithm to compute $u$ at a point $x_0$ on the line linking the centers of the two atoms, at a
distance $1.5\usk\angstrom$ from the closest center. We used the \texttt{SNJ} jump method, and the \texttt{ANJ} jump method for
$\alpha=3$ and $\alpha=10$.

The reference value involved in the error estimate cannot be computed as above because the spherical symmetry is broken. We could use
APBS to solve the nonlinear Poisson-Boltzmann PDE, but our tests in the linear case show that the adaptive finite element method of
this solver can produce small errors which could prevent us from observing the convergence of our methods. This is why we use a
reference value computed with the \texttt{ANJ} jump method ($\alpha=10$) with a large number of Monte-Carlo simulations ($10^6$),
$h=0.001$ and $\epsilon=10^{-6}$.

We tested several values of $a$. The results are very similar to those obtained in the case of a single atom.

\begin{figure}[ht]
  \centering
  \includegraphics[width=.45\textwidth]{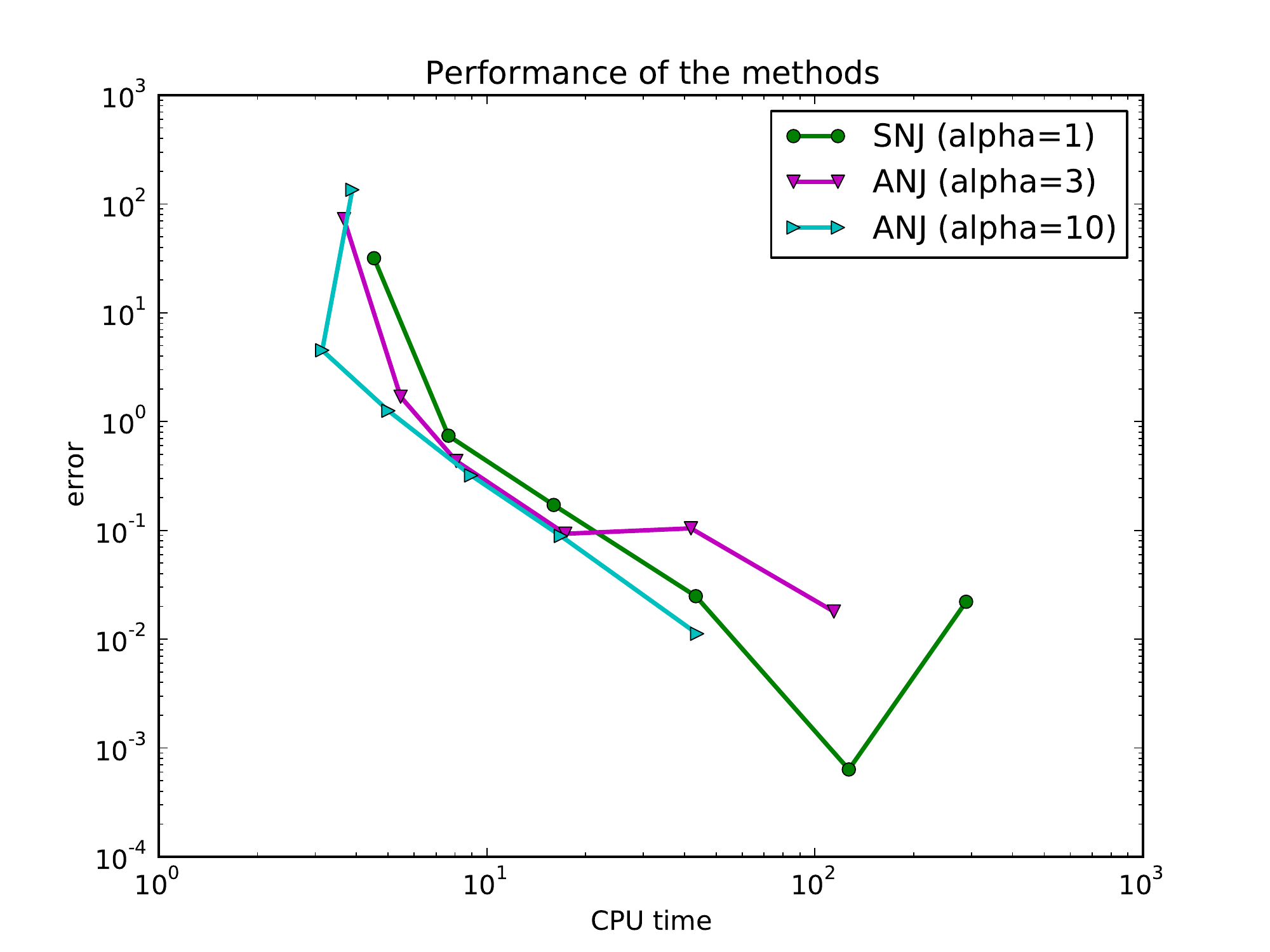}
  \caption{Performance of the branching algorithm with 3 jump methods for two atoms with $a=0.2\usk\angstrom$: error (Log scale) as a
    function of CPU time (Log scale).}
  \label{fig:perform-2-atoms-a=0.2}
\end{figure}

The results for a small value $a=0.2\usk\angstrom$ are shown in Figure~\ref{fig:nonlin-2-atoms-a=0.2}. As in the case of a single atom
for $z=0.2$, the algorithm behaves nicely. We observe as before small confidence intervals and a convergence of the error to 0 at a
speed of the order of $h$. The performance plot of Figure~\ref{fig:perform-2-atoms-a=0.2} shows similar performances for the three
values of $\alpha$, with a slight advantage for the largest value of $\alpha$, which gives the better error for a given CPU time and
an appropriate value of $h$.

The value $a=0.5\usk\angstrom$ also has a relatively good behavior, shown in Figure~\ref{fig:nonlin-2-atoms-a=0.5}, although the
convergence is not as good as for $a=0.2\usk\angstrom$, similarly as for $z=1$ in the case of a single atom. In particular, the
convergence for small $h$ is not as clear as for $a=0.2\usk\angstrom$, but, again, for $h$ between $0.003$ and $0.1$, the relative
error of the method is smaller than $2$\%.

As in the case of a single atom, higher values of $a$ lead to larger values of the variance and the error of the algorithm and the
convergence fails.

\begin{figure}[ht]
  \centering
  \subfigure[Monte-Carlo average of the 3 jump methods as a function of $h$ (Log
  scale).\label{fig:average-nonlin-2-atoms-a=0.5}]{\includegraphics[width=.45\textwidth]{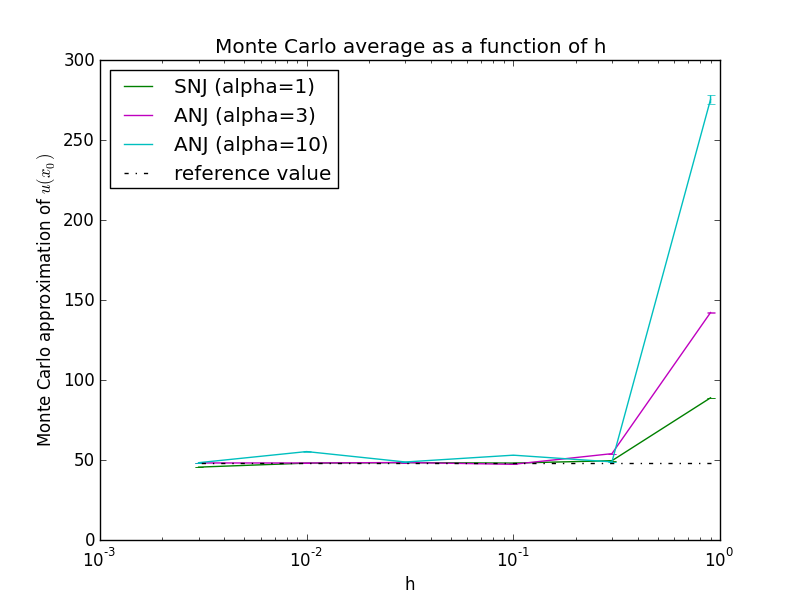}}\quad
  \subfigure[Error (Log scale) of the 3 jump methods as a function of $h$ (Log scale).\label{fig:error-nonlin-2-atoms-a=0.5}]{
    \includegraphics[width=.45\textwidth]{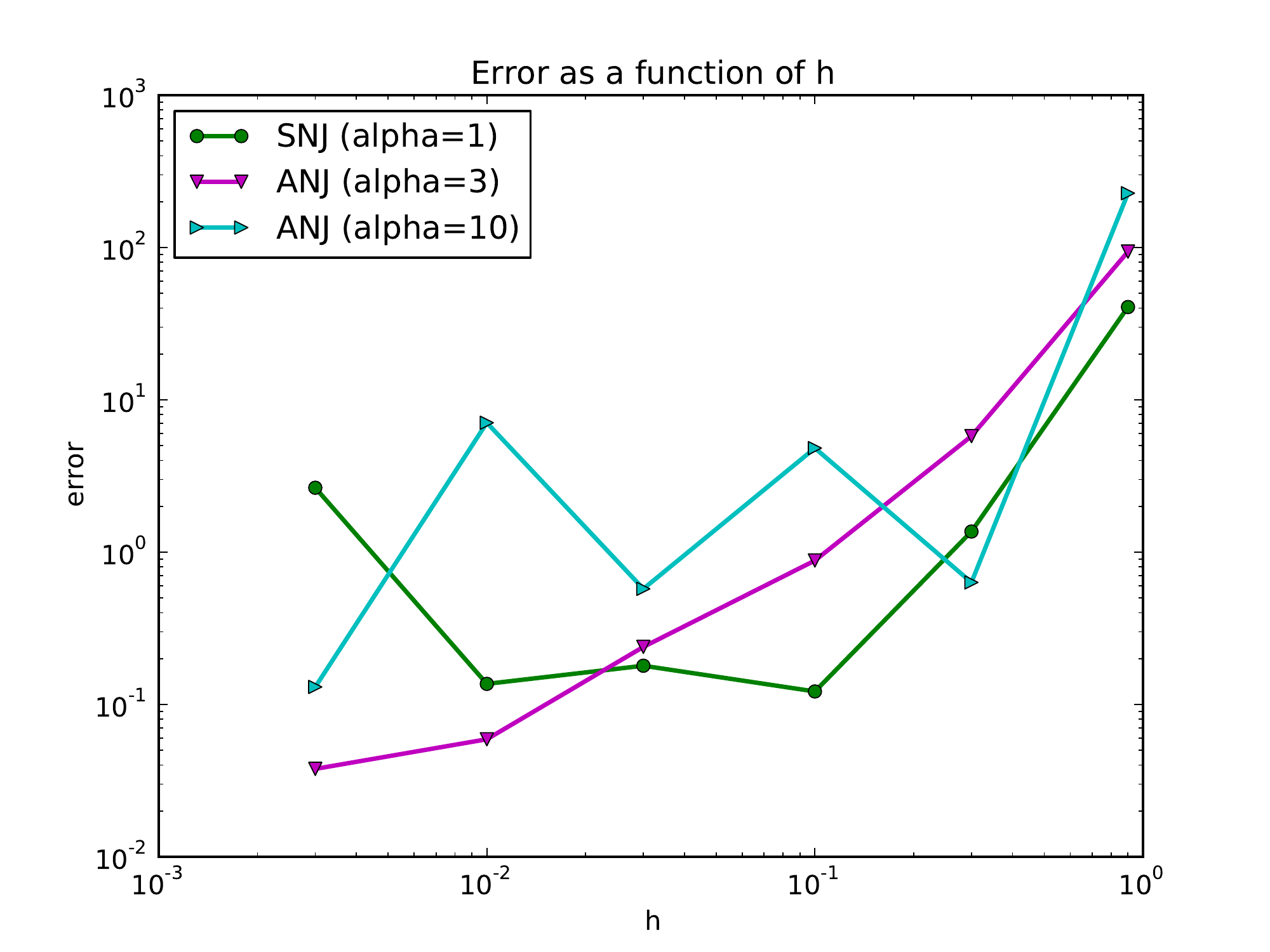}}
  \caption{Convergence and error of the branching algorithm with jump methods \texttt{SNJ} and \texttt{ANJ} ($\alpha=3$ and
    $\alpha=10$) in the case of two atoms with distance $a=0.5$.}\label{fig:nonlin-2-atoms-a=0.5}
\end{figure}

\section{Conclusion and perspectives}
\label{sec:ccl}

Our numerical experiments on the linear case show that the \texttt{TAJ} jump method can be used without significant increase in computational time,
and with a slightly improved expected error. Therefore, it allows to take a larger value of $h$ for a given error threshold, hence actually
reducing the computational time. This is a new argument which, together with those developed in~\cite{mascagni-al-13}, allows to
expect that optimized walk-on-spheres Monte-Carlo solvers for the linear Poisson-Boltzmann equation can be made competitive in terms
of computational time with respect to classical deterministic methods.

Our preliminary tests to solve the nonlinear Poisson-Boltzmann PDE using branching particle systems show that our method has roughly
equivalent performances than the walk-on-spheres solver in the linear case (with \texttt{SNJ} and \texttt{ANJ} methods). However, in
some situations, typically when the electrostatic potential $u$ is large on $\Gamma$, the variance of the method might explode. This
issue requires to develop adequate variance reduction techniques, to be discussed in future work. We have already
tested a stratification technique, which is quite efficient in reducing the variance of the method and makes it converge in cases
where the unstratified algorithm shows variance explosion. However, this method fails again for too large values of $u$ on $\Gamma$.

First, a deeper theoretical analysis of the variance of the algorithm is needed. In particular, a study of the influence of the
parameters and the boundary conditions (see Remark~\ref{rem:variance}) on the variance can give insights on adequate variance
reduction techniques or on appropriate values of the parameters $\lambda,\ g(k),\ p_k$ for the stratified Monte-Carlo method. The
parameter values \eqref{eq:choix-param} proposed in Section~\ref{sec:first-choice} are one possible choice, but other possibilities might be considered and tuned in
order to optimize the variance.

Additional variance reduction techniques have to be explored. For example, one could try to reduce the variance of the score within
each stratum. To undertake this, we can analyze the probabilistic interpretation~\eqref{eq:interpr-proba-ext}. Let us denote by $X$
the r.v. inside the expectation in the r.h.s. of this equation. $X$ might be 0 if one of the particles dies without children before
leaving $\Omega_{\text{out}}$, or might be a product of values of $u$ on $\Gamma$ if all the particles hit $\Gamma$ before dying. If
$u$ takes large values on $\Gamma$, this product might be very large, and hence the variance of $X$ is very large. One could reduce
this variance by reducing the probability that $X=0$. This could be done using importance sampling techniques, for example by adding
to the Brownian motion of each particle a drift towards the center of the molecule.

We can also study pruning techniques in the spirit of~\cite{blomker-al-07}, where pruning of genealogical trees of branching particle
systems is used to study the probabilistic interpretation of the Fourier transform of Navier-Stokes equation.

\medskip\medskip\noindent{\large \sc Acknowledgments. }\\
\indent The authors thank  S{\'e}lim Kraria (from DREAM, INRIA Sophia Antipolis M{\'e}diterran{\'e}e)  and Pierre Navarro (from IRMA,
Universit{\'e} de Strasbourg) for their precious help regarding the software development aspect.

This work was granted access to the HPC resources of Aix-Marseille Universit{\'e} financed by the project Equip@Meso (ANR-10-EQPX-29-01)
of the program ``Investissements d'Avenir'' supervised by the {Agence Nationale pour la Recherche}.

\appendix
\section*{Appendix}

\section{Values of constants}\label{appendixA}

The following table gives the values of the  constants involved in the Poisson-Boltzmann PDEs we consider in this work. 
\begin{table}[h!]
\centering
\begin{tabular}{|l|c|l|}
\hline
 & symbol & value \\
\hline
Boltzmann constant & $k_B$ & $1,3806488 \times 10^{-23}\usk\squaren\meter\usk\kilogram\usk\rpsquare\second\usk\reciprocal\kelvin$ \\ 
\hline
Charge of an electron & $e_c$& $1,602176565\times10^{-19}\usk\coulomb$\\
\hline
Temperature & $T$ & $298 \usk\kelvin$\\
\hline
Vacuum permittivity & $\varepsilon_0$ & $ 8,854187817 \times  10^{-12}\usk\farad\usk\reciprocal\meter$\\
\hline
Avogadro constant & $\mathcal{N}_A$ & $6,02214129 \times 10^{23} \usk\reciprocal\mole$\\
\hline
Molecule relative permittivity & $\varepsilon_{\text{in}}$ & 2$\phantom{0}$ $\qquad$ (dimensionless)\\
\hline
Solvent relative permittivity & $\varepsilon_{\text{out}}$ & 80  $\qquad$ (dimensionless)\\
\hline
Solvent ion concentration & $c$ & $1\usk\mole\usk\reciprocal\liter$\\
\hline
Solvent ion relative charge & $z$ & $\pm 1\,\qquad$ (dimensionless)\\
\hline
Inverse Debye length~\eqref{eq:kappa} & $\bar\kappa$ & $2.9132\usk\reciprocal\angstrom$ \\\hline
\end{tabular}
\caption{The physical constants  used  in the simulations in the the International System of Units.}
\end{table}

\bibliographystyle{plain}
\def\cprime{$'$} \def\cprime{$'$} \def\cprime{$'$}

\end{document}